\documentclass[12pt, reqno]{amsart}

\usepackage{textcomp} 
\usepackage{amsmath, amsfonts,amsthm,amssymb,amsxtra}
\usepackage{amscd}
\usepackage{enumitem}

\usepackage[dvipsnames]{xcolor}

\usepackage{tikz-cd} 
\usepackage{todonotes}
\usepackage{soul} 

\usepackage{graphicx}
\usepackage{verbatim}

\newtheorem{theorem}{Theorem}[section]
\newtheorem{proposition}[theorem]{Proposition}
\newtheorem{lemma}[theorem]{Lemma}

\newtheorem{corollary}[theorem]{Corollary}
\newtheorem{question}[theorem]{Question}

\theoremstyle{definition}

\newtheorem{definition}[theorem]{Definition}
\newtheorem{examples}[theorem]{Example}

\theoremstyle{remark}

\newtheorem{remark}[theorem]{Remark}

\numberwithin{equation}{section}

\renewcommand{\epsilon}{\varepsilon}

\renewcommand{\phi}{\varphi}
\newcommand{\R}{\mathbb{R}}

\newcommand{\Z}{\mathbb{Z}}

\newcommand{\p}{\partial}

\begin{document}
\title{Weinstein presentations for high-dimensional antisurgery} 

\begin{abstract}
In this paper, we give an algorithm for describing the Weinstein presentation of Weinstein subdomains obtained by carving out regular Lagrangians. Our work generalizes previous work in dimension three and requires a novel Legendrian isotopy move (the ``boat move") that changes the local index of Reeb chords in a front projection.  As applications, we describe presentations for certain exotic Weinstein subdomains  and give explicit descriptions of $P$-loose Legendrians.
\end{abstract}

\date{\today}

\author[I. Datta]{Ipsita Datta} 
\address{Department of Mathematics, ETH Z\"urich, Z\"urich, Switzerland} \email{ipsita.datta@math.ethz.ch}

\author[O. Lazarev]{Oleg Lazarev} 
\address{Department of Mathematics, University of Massachusetts Boston, Boston, MA, USA} \email{oleg.lazarev@umb.edu}

\author[C. Mohanakumar]{Chindu Mohanakumar} 
\address{Mathematics Department, Fordham University, New York, NY, USA} \email{cmohanakumar@fordham.edu}

\author[A. Wu]{Angela Wu} 
\address{Department of Mathematics, Louisiana State University, Baton Rouge, LA, USA} \email{awu@lsu.edu}

\maketitle

\section{Introduction}\label{sec: introduction}

Weinstein domains \cite{Weinstein_91_CSSH} are exact symplectic manifolds equipped with symplectic handlebody decompositions, analogous to CW complexes in topology. These domains are relatively easy to construct by consecutively attaching handles along isotropic spheres in contact manifolds. Weinstein presentations or diagrams keep track of these isotropic spheres and their interactions with each other and make computation of invariants, like the wrapped Fukaya category, tractable via surgery formulas and gluing formulas \cite{Bourgeois_Ekholm_Eliashberg_surgery, Ganatra_Pardon_Shende_descent}. 

A wealth of symplectically exotic Weinstein domains can be constructed as \textit{subdomains} of more standard Weinstein domains, obtained by \textit{carving out} Lagrangian disks. 
For example, Sylvan and the second author \cite{Lazarev_Sylvan_2023_PLWS} showed that if $n \ge 5$, the standard cotangent bundle $T^*S^n$ has infinitely many Weinstein subdomains that are diffeomorphic to $T^*S^n$ but pair-wise non-symplectomorphic. They also constructed $P$-loose Legendrians as subdomains of the sector $T^*D^n$
and showed that these $P$-loose Legendrians are smoothly isotopic but not Legendrian isotopic. The contact analog of carving out Lagrangian disks---\textit{contact antisurgery}---is important for the construction of contact structures; for example, any contact structure on $S^{2n-1}$ is obtained by doing a single contact surgery and antisurgery on the standard contact structure $(S^{2n-1}, \xi_{std})$ \cite{Lazarev_2020_MCSS}. Weinstein subdomains are also attractive from the point of view of categorical invariants; their wrapped Fukaya categories are localizations of the Fukaya category of the ambient domain by the localization formula in \cite{Ganatra_Pardon_Shende_descent}.
Finally, any Weinstein domain deformation retracts to its singular Lagrangian skeleton; therefore the question of studying Weinstein subdomains of a fixed domain $X$ is precisely the question of finding singular Lagrangian skeleta in $X$.

Weinstein subdomains also arise naturally when relating complements of toric divisors and their (partial) smoothings. That is, $X\setminus D$ is a Weinstein subdomain of $X\setminus \tilde{D}$ for a Weinstein domain $X$, divisor $D \subset X$, and smoothing $\tilde{D}$ of $D$.  In 4-dimensions, the Weinstein presentations of such manifolds have been related in this context by Acu, Capovilla-Searle, Gadbled, Marinkovic, Starkston, and the fourth author \cite{ACSGNNSW_22_WHCSTD}. They define a necessary condition on a Delzant polytope of the toric manifold, which ensures that the complement of a corresponding partial smoothing of the toric divisor supports a Weinstein structure. They give an algorithm to construct an explicit Weinstein presentation for the complement of such a partially smoothed toric divisor. 

On the other hand,  Weinstein presentations have not been described yet for the constructions in \cite{Lazarev_Sylvan_2023_PLWS}; nor does there exist a general procedure for describing explicit Weinstein presentations for general Weinstein subdomains.  For instance, the construction of $P$-flexible Weinstein manifolds was relatively inexplicit due to the fact that it was not clear how the carving out/antisurgery modified the front of the original Legendrian. In particular, the front projection of these $P$-loose Legendrians was not known.

In this paper, our goal is to remedy this situation. We focus on the problem of constructing explicit Weinstein presentations of exotic Weinstein subdomains constructed by carving out Lagrangian disks. We want the presentation of such a subdomain to be in terms of a Weinstein presentation of the original Weinstein domain, which is compatible with the Lagrangian disk. In the process, we introduce a new Legendrian isotopy move, called the boat move.

As a concrete application, we give an explicit front projection for the $P$-loose Legendrians constructed indirectly in \cite{Lazarev_Sylvan_2023_PLWS}.
For any collection of integers $P$, 
a Legendrian $\Lambda$ is said to be \emph{$P$-loose} if it is isotopic to 
$\Lambda \sharp \Lambda_P$, the connected sum of $\Lambda$ and $\Lambda_P$
where $\Lambda_P$ is a $P$-loose Legendrian unknot, defined in $\mathbb{R}^{2n+1}$ for $n \ge 4$. 
This operation of taking connect sum with $\Lambda_P$ can be used to make $P$-loose Legendrian representatives of any smooth $n$-dimensional knot type. A Weinstein manifold constructed via handle attachments along $P$-loose Legendrians is called \emph{$P$-flexible}. In \cite{Lazarev_Sylvan_2023_PLWS}, it was shown that $P$-loose Legendrians have properties that generalize those of loose Legendrians, which were introduced by Murphy in \cite{Murphy_12_LLEHD}. If $0 \in P$, then $\Lambda_P$ is a loose Legendrian unknot. In general, $\Lambda_P$ is not necessarily loose but has Legendrian dga (with loop space coefficients) equal to $\Z[\frac{1}{P}]$, see Section~\ref{sec: main construction}. Furthermore, for any Legendrian $\Lambda \subset Y$, the Chekanov-Eliashberg dga satisfies 
$$	
CE(\Lambda \sharp \Lambda_P) \cong CE(\Lambda)[ P^{-1}].
$$
and hence vanishes with  $\mathbb{Z}/P\mathbb{Z}$ coefficients.

\subsection{Main results} 
For any Weinstein subdomain $X_0$ of a domain $X$, the complement $X\setminus X_0$ admits the structure of a Weinstein cobordism $C$. This cobordism has a subcritical part $C_{sub}$ that does not change invariants like the Fukaya category, and some critical handles $H_i^n$, $i = 1, \dots, l$, with Lagrangian co-cores disks $L_i^n \subset X$. Hence $X_0$ can be described, up to subcritical cobordism, as $X\setminus \left( \cup_{i=1}^l L_i \right)$. 

Conversely, given any Lagrangian disk $L \subset X$, $X\setminus L$ is an exact subdomain of $X$.
The contact boundary $\partial(X\setminus L)$  of $X \setminus L$ is obtained from the contact boundary $\partial X$ of $X$
by doing \emph{antisurgery}, or $(+1)$-contact surgery, along the Legendrian sphere $\partial L \subset \partial X$. 
To ensure that $X\setminus L$ is Weinstein, we assume that $L$ is \textit{regular} \cite{Eliashberg_Flexible_Lagrangians}, which implies that our starting Weinstein presentation for $X$ is compatible with $L$, as we describe later.

The main goal of this paper is to give explicit constructions---Weinstein presentations and handlebody decompositions---for Weinstein manifolds obtained via antisurgery.
To do this we introduce a new family of $n$-dimensional Legendrian moves in the front projection. We construct Legendrian isotopies called \emph{$D^k$-suspensions} that consist of a Legendrian isotopy $\psi$ of an $(n-k)$-dimensional slice suspended over a $k$-dimensional disk. 
\begin{proposition}\label{prop: boat_move_intro}
        Given a Legendrian isotopy $ \psi: D^{n-k} \times [0,1]_t \to \R^{2(n-k) + 1}$ which is the identity near $\partial D^{n-k}$ and t-independent near $\partial [0,1]$, its
        $D^k$-suspension, $\Sigma_{D^k} \{\psi\}$ is a Legendrian in $\R^{2n+1}$, and is Legendrian isotopic to $D^k \times D^{n-k}$ relative to the boundary.
    \end{proposition}
We call the $D^k$-suspension of a $(n-k)$-dimensional Reidemeister 1 move an \emph{$(n,k)$-boat move}, see for instance Figure~\ref{fig:surface_boat_move}. 
Using these boat moves, we can construct our desired presentations for Weinstein domains obtained via antisurgery.

In dimension 3, there are several existing results, for example \cite{Ding_Geiges_09_HMCSD}, explaining how to do antisurgery along Legendrian circles that admit Lagrangian disk fillings. Our main contribution is that we are able to replicate such explicit constructions in higher dimensions by first using $(n,k)$-boat moves.

Next we state our main result. Let $L \subset X$ be a regular Lagrangian disk in a Weinstein domain $X^{2n}$.
By \cite{Eliashberg_Flexible_Lagrangians}, any regular Lagrangian disk $L \subset X$ in a Weinstein domain $X$ can be presented as $$D^n \subset T^*D^n \cup \left( \cup_i H_i\right),$$ where $H_i$ are Weinstein handles attached to $T^*D^n$ in the complement of $\partial D^n \subset T^*D^n$. For simplicity, we assume here that all of these handles have index $n$, although this assumption can be removed. 
Let $$\Lambda_\emptyset \cup \left( \cup_i \Lambda_i \right) \subset \mathbb{R}^{2n-1}$$
be the Legendrian link formed by the following Legendrians. First, $\Lambda_\emptyset = \partial D^n$ is the standard Legendrian unknot with front projection in $\mathbb{R}^{n}$  given by the ``flying saucer" (with a $S^{n-2}$-family of cusps and no other singularities). Second,  $\Lambda_i$ are the attaching Legendrians of $H_i$. Let $C_i$ denote the set of Reeb chords from $\Lambda_i$ to $\Lambda_\emptyset$ with front projection contained in the subset of $\R^{n}$ bounded by the flying saucer front projection $\pi(\Lambda_\emptyset) \subset \mathbb{R}^{n}$, and let $C = \cup_i C_i$; see Figure~\ref{fig:acceptable_Reeb_chords}. We assume all such chords are non-degenerate, and furthermore correspond to critical points of a local Morse function whose indices we call the local index of the Reeb chord; in particular, $C$ is finite. 

\begin{figure} [h]
	\centering
  \begin{tikzpicture}
    \node at (0,0) {\includegraphics[width=5.2cm]{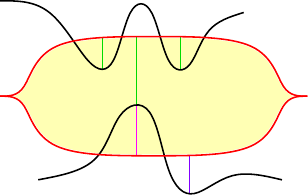}};
    \node at (-2.5,1.2) {$\Lambda_i$};
    \node[color=red] at (-2.5,-0.5) {$\Lambda_\emptyset$};
    \end{tikzpicture} 
	\caption{The green Reeb chords are bounded by the flying saucer and go from the black $\Lambda_i$ to the red $\Lambda_\emptyset$; these have critical points on $\Lambda_i$ with local index $1, 0, 1$ for the height difference Morse function from $\Lambda_i$ to $\Lambda_\emptyset$.  
We will apply $(1,0), (1,1), (1,0)$ boat moves at these critical points respectively, which correspond to doing a Reidemeister 1 move at the two index 1 critical points and nothing at the middle index 0 critical point. 
The purple Reeb chord is not bounded by the flying saucer. The pink Reeb chord is bounded by the flying saucer but does not go from $\Lambda_i$ to $\Lambda_\emptyset$.
 } 
	\label{fig:acceptable_Reeb_chords}
\end{figure}

\begin{theorem}\label{thm: general_antisurgery_intro}
 There is a  Weinstein presentation for the Weinstein subdomain $X \setminus L \subset X$ with the following properties: 
  \begin{itemize}
\item The Weinstein presentation of $X \setminus L$ has one more $(n-1)$-handle than the Weinstein presentation for $X$.
\item The $n$-handles for $X\setminus L$ are in one-to-one correspondence with the $n$-handles of $X$.
\item The attaching sphere $\Lambda_i'$ of the $n$-handle $H'_i$ of $X\setminus L$ is obtained from the attaching sphere $\Lambda_i$ of the corresponding handle $H_i$ of $X$ in the following way: for each Reeb chord $\gamma$ in $C_i$
with local index $k$, we apply an $(n, n-k)$-boat move to $\Lambda_i$ and do a cusp connected sum with a Legendrian that goes through the new $(n-1)$-handle one time.
  \end{itemize}
\end{theorem} 

\begin{remark}
The assumptions for Theorem \ref{thm: general_antisurgery_intro} can be weakened. For example, as explained in Lemma \ref{lemma:non-degenerate_Reeb_chords}, 
any Legendrian link in $\mathbb{R}^{2n+1}$ can be perturbed by a $C^0$-small isotopy so that all Reeb chords in $C$ are non-degenerate and correspond to critical points of a Morse function; however, perturbing a degenerate Reeb chord may result in a larger (finite) number of non-degenerate Reeb chords. 

Also, the assumption that $X$ takes form $T^*D^n \cup (\cup_i H^n_i)$ can be generalized to allow subcritical handles. Generically, there will be no Reeb chords between $\Lambda_\emptyset$ and the attaching spheres of these subcritical handles, but the attaching spheres of the critical handles may interact with these subcritical handles. Hence, only a portion of these attaching spheres map to the contact boundary $\partial_\infty T^*D^n$, resulting in a Legendrian subset $\Lambda_i \subset \partial_\infty T^*D^n$ (which are not necessarily spheres but manifolds with boundary). In that case, a generalization of Theorem~\ref{thm: general_antisurgery_intro} holds by considering Reeb chords between $\Lambda_\emptyset$ and the subset $\Lambda_i$.
\end{remark}

Additionally, as a main application of Theorem~\ref{thm: general_antisurgery_intro}, we construct explicitly the $P$-loose Legendrians in the Weinstein presentations of $P$-flexible Weinstein domains.

\begin{corollary}\label{cor: explicit_P_loose_intro}
For integers $p \ge 0$ and $n \ge 2$, the $P$-loose Legendrian unknot, $\Lambda_P \subset \mathbb{R}^{2n+1}$, for $P = \{p\}$, is a Legendrian submanifold consisting of four loose Legendrian unknots which are completely parallel away from a bounded region. Within this bounded region they may be linked (in a way that depends on p) and are connected via three boat moves and cusp connected sum gluings (see Figure~\ref{fig:p_loose}); the front projection in this bounded region has no singularities except for the singularities in the local boat moves. 
\end{corollary}

\begin{figure}
    \centering
    \includegraphics[width=6.35 cm]{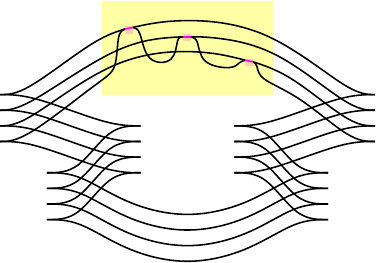}
    \caption{A cartoon depicting a low-dimensional slice of a $P$-loose Legendrian which cuts through the three boat move cusp connect sums, represented by the pink regions. Away from a neighbourhood of this slice, the black Legendrians may be linked in the yellow region but are parallel outside of it.} 
    \label{fig:p_loose}
\end{figure}

\subsection{Structure of the paper}

In Section~\ref{sec: background}, we give some background on antisurgery and $P$-loose Legendrians. In Section~\ref{sec: Legendrian_moves}, we discuss the boat move and prove Proposition~\ref{prop: boat_move_intro}. In Section~\ref{sec: main construction}, we present our construction for producing Weinstein presentations for subdomains and prove Theorem~\ref{thm: general_antisurgery_intro} and Corollary~\ref{cor: explicit_P_loose_intro}. We then use our construction to give explicit examples of Weinstein antisurgery manifolds and conclude with some open questions.

\subsection{Acknowledgements}
The authors would like to thank the organizers (Orsola Capovilla-Searle, Roger Casals, and Caitlin Leverson) of the SYNC Early Career Workshop at the University of California, Davis where this project was started in August, 2022. The authors would also like to thank Hiro Lee Tanaka, Lea Kenigsberg, Tonie Scroggin, Georgios Dimitroglou Rizell, Jonathan Michala, and Josh Sabloff for helpful discussions. 
ID was supported by NSF grant DMS-1926686. OL was supported by NSF grant DMS-2305392. CM was supported by NSF Grant DMS-2003404. AW was supported by NSF grant DMS-2238131. 

\section{Background}\label{sec: background}

\subsection{Contact surgery and Weinstein manifolds} Let $M$ be an $n$-dimensional manifold. In the topological setting, we perform $k$-surgery on $M$, for $k \leq n$, by first removing a tubular neighborhood $N(S^k) \cong S^k \times D^{n-k}$ of a $k$-sphere, and then gluing back a copy of $D^{k+1} \times S^{n-k-1}$ along boundary $S^{n-k} \times S^{n-k-1}$. A choice of framing, that is, the particular identification $\phi: N(S^k) \cong S^k \times D^{n-k}$, specifies our gluing, which then determines our surgered manifold up to diffeomorphism.

When the $k$-sphere is a Legendrian sphere, $\Lambda$, inside $M$, a contact manifold, there is a canonical framing, provided an identification $\Lambda \cong S^k \subset \mathbb{R}^{k+1}$. To see this, note that $TM = T\Lambda \oplus \text{\textlangle}R_{\alpha}\text{\textrangle}\oplus J(T\Lambda) $, and so $N(\Lambda) \cong  \text{\textlangle}R_{\alpha}\text{\textrangle}\oplus J(T\Lambda)$ for a contact form $\alpha$ and corresponding Reeb vector field $R_\alpha$. The bundle $\text{\textlangle}R_{\alpha}\text{\textrangle}\oplus J(T\Lambda)$ can be identified with the stabilized tangent bundle of $\Lambda$, which carries a canonical trivialization after an identification of $\Lambda$ with $S^k \subset \mathbb{R}^{k+1}$ --- this trivialization is the canonical framing. At this point, we drop the dimension $k$ of our sphere from the notation, and refer to such a surgery by $(\pm \frac{m}{k})$-surgery, where $\pm \frac{m}{k}$ is a fraction relating our choice of framing to the canonical framing. For example, attaching a critical Weinstein handle affects the contact boundary by a $(-1)$-surgery.

\subsection{Legendrian moves}\label{sec:legendrian moves}
Legendrian moves refer to the replacement of a Legendrian by something Legendrian isotopic that differs from it only within a Darboux neighbourhood. These are often depicted via front diagrams.

Legendrian Reidemeister moves are analogous to knot diagram moves which preserve the topological knot type. They usually refer to replacements in the front diagrams of Legendrian knots in a contact $3$-manifold depicted in Figure~\ref{fig:R-moves}. These fully characterize Legendrian isotopies for $1$-dimensional Legendrians in contact $3$-manifolds.
\begin{theorem}[See \cite{Swiatowski_1992}]
    Two front diagrams represent Legendrian isotopic Legendrian knots if and only if they are related by regular homotopy and a sequence of moves shown in Figure~\ref{fig:R-moves}.
\end{theorem}

\begin{figure}
    \centering
    \includegraphics[width=5.3 cm]{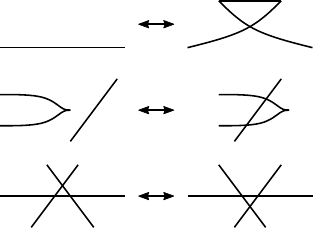}
    \caption{The Legendrian Reidemeister moves.} 
    \label{fig:R-moves}
\end{figure}

\begin{figure}
    \centering
    \includegraphics[width=14.7 cm]{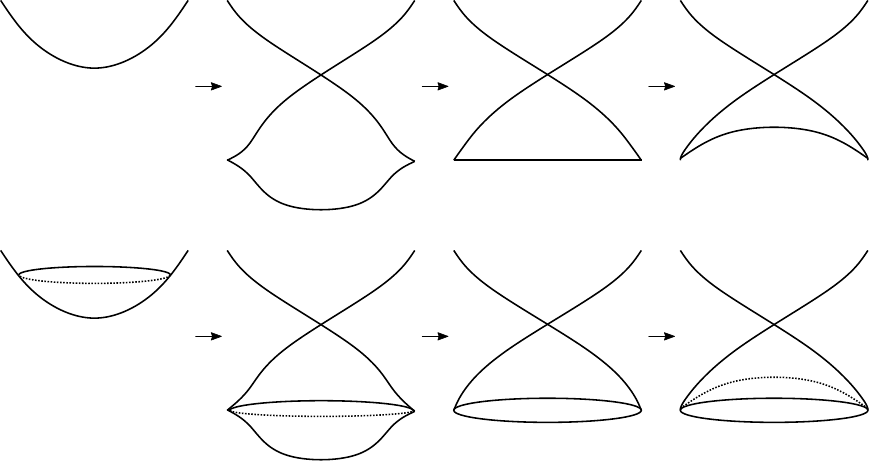}
    \caption{A Legendrian Reidemeister 1 and isotopy for a knot (top) and a surface (bottom).} 
    \label{fig:min_to_max}
\end{figure}

We will also be using higher dimensional first Reidemeister moves. For any Legendrian submanifold $\Lambda \subset (Y, \xi)$, the Legendrian submanifold $\Lambda'$ obtained by replacing a graphical portion of the front of $\Lambda$ (with respect to a particular Darboux chart) by the rightmost front depicted in Figure~\ref{fig:min_to_max} is Legendrian isotopic to $\Lambda$. To see this note that the first arrow is the same as one of the higher dimensional first Reidemeister moves described in \cite{Bourgeois_Sabloff_Traynor_2015_LCGFCG}. The subsequent two arrows of Figure~\ref{fig:min_to_max} are front diagrams of Legendrian isotopies where the domed part of the front is pushed inwards via an isotopy of the $(x_1, \dots, x_n, z)$-plane that does not have vertical tangencies at all times.  We refer to this replacement within a Darboux chart as a {\bf $k$-dimensional first Reidemeister move}.\footnote{Replacing the figures in Figure~\ref{fig:min_to_max} by their reflections about the $(x_1, \dots, x_n)$-plane gives analogous Legendrian moves. We omit these in the discussion for simplicity.} 
Let 
$$R1_k : [-1,1] \times D^k \to \R^{2k+1}$$ 
denote the associated Legendrian isotopy.
Note that the fronts depicted in Figure~\ref{fig:min_to_max} have $S^{k-2}$-symmetry about the $z$-axis passing through the point of singularity.

\begin{remark}
    The fronts depicted in Figure~\ref{fig:min_to_max} and therefore, fronts obtained whenever we apply the n-dimensional R1, are not generic \cite{Dimitroglou_Rizell_11_KLSFRC} as there is a singularity at the cone point. This does not matter for our purposes but is important to keep in mind especially if computing Legendrian Contact Homology.
\end{remark}

\begin{remark}
    Higher dimensional versions of the second and third Legendrian Reidemeister moves also exist, see for instance their depictions in Section 2.4 of \cite{Casals_Murphy_2019_LFAV}. In higher dimensions, these three moves in the front do not fully characterize all Legendrian isotopies. Generally, we can isotope Legendrian fronts past each other by the Reeb flow, as long as there are no Reeb chords between them.    
\end{remark}

In addition to Reidemeister moves, we will use Legendrian handle slides and the addition or removal of a cancelling pair of Weinstein handles in our diagrammatic calculus. These are moves that allow us to pass between different Weinstein handle diagrams of equivalent Weinstein domains. In other words, before and after these moves, the Legendrians depicted are isotopic in the surgered contact boundaries (but maybe not isotopic in the original contact boundaries). 

Following Casals and Murphy \cite{Casals_Murphy_2019_LFAV}, we depict the effect of Legendrian handle slides in Figure~\ref{fig:handleslides}. A handle slide over a $(+1)$-surgery curve produces a cone singularity, while sliding over a $(-1)$-surgery produces a circle of cusp singularities. Meanwhile, for a Weinstein manifold of dimension $2n$, the addition or removal of a cancelling pair refers to adding in an $n$-handle and an $(n-1)$-handle such that the attaching sphere of the $n$-handle and the belt sphere of the $(n-1)$-handle intersect exactly once.

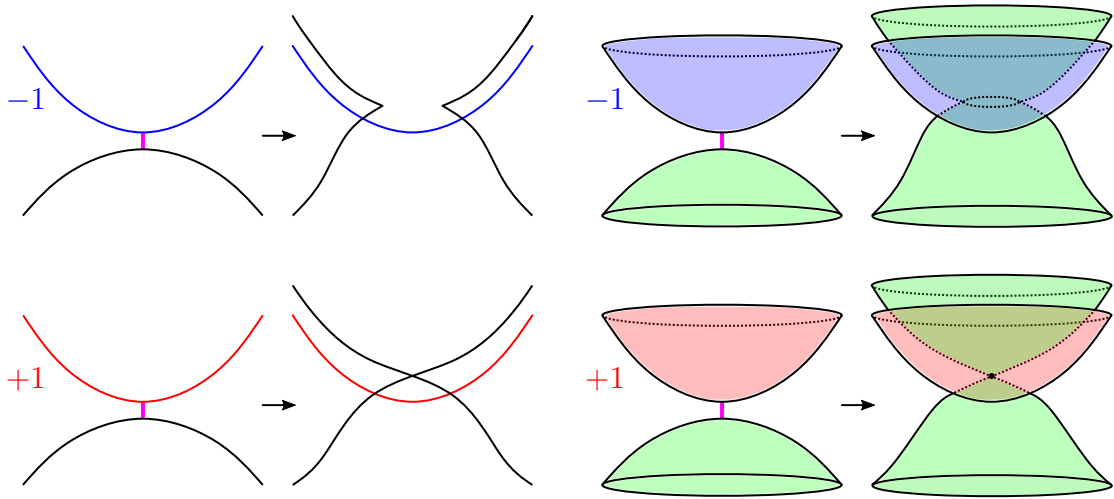
\begin{figure}
    \centering
    \begin{tikzpicture}
    \node at (0,0) {\includegraphics[width=14.5 cm]{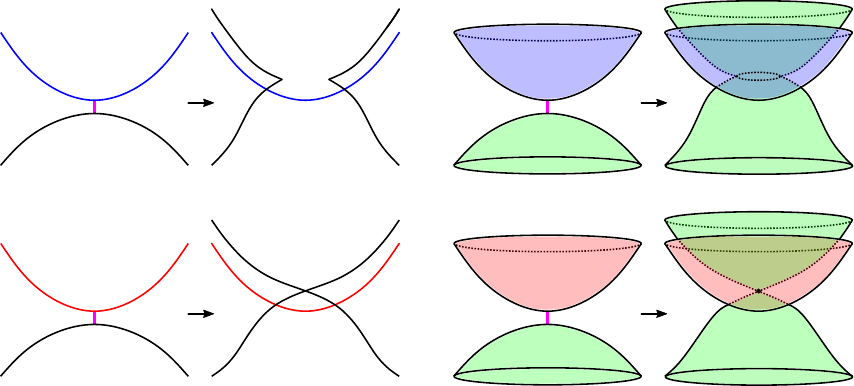}};
    \node[color=blue] at (-7.2,2) {$-1$};
    \node[color=red] at (-7.2,-1.7) {$+1$};
    \node[color=blue] at (0.5,2) {$-1$};
    \node[color=red] at (0.5,-1.7) {$+1$};
    \end{tikzpicture}
    \caption{Handle slides over (-1) and (+1) Legendrians (in blue and red respectively) in 3 and 5 dimensions.} 
    \label{fig:handleslides}
\end{figure}

\subsection{Loose and $P$-loose Legendrians, flexible and $P$-flexible domains}\label{sec:flexibility}

In \cite{Murphy_12_LLEHD}, Murphy introduced a class of Legendrians called loose Legendrians. These Legendrians are characterized by an explicit local model.

\begin{definition}
A loose unknot $\Lambda_l \subset (\mathbb{R}^{2n+1}, \xi_{std})$ for $n \ge 2$ is a Legendrian that is formally isotopic to the Legendrian unknot and there is an $\mathbb{R}^3$ slice that intersects $\Lambda_l$ transversely and the front projection of $\Lambda_l \cap \mathbb{R}^3$ is the one-dimensional stabilized arc (see Figure~\ref{fig:loose_chart}). 
A Legendrian $\Lambda \subset (Y, \xi)$ is loose if it is Legendrian isotopic to $\Lambda \sharp \Lambda_l \subset Y \sharp \R^{2n+1} \cong Y$. 
\end{definition}

\begin{figure}
    \centering
    \includegraphics[width=1.5 cm]{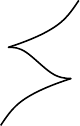}
    \caption{A one-dimensional stabilized Legendrian arc.} 
    \label{fig:loose_chart}
\end{figure}

\begin{remark}
Since we prefer to work with closed Legendrians, the above definition involves loose Legendrian unknots (which are spheres) instead of the original definition which involves loose Legendrian \textit{charts}, which are Legendrian disks; these definitions are equivalent.
\end{remark}

We emphasize that there is not a canonical model for the loose Legendrian unknot (except in dimension 1, which must be excluded for reasons explained below). For example, we can take any codimension zero subdomain $U$ near the cusp of the Legendrian unknot $\Lambda_\emptyset$ and push through $U$ past the cusp to create $\Lambda_U$ (see Figure~\ref{fig:push_to_loose}). 
Then $\Lambda_U$ is always loose and if $U$ has Euler characteristic zero, then $\Lambda_U$ is Legendrian isotopic to the loose Legendrian unknot. 
Alternatively, one can take any \textit{closed} codimension $1$ submanifold $V$ of $\Lambda_\emptyset$ and add a $1$-dimensional zig-zag spun by $V$ to create $\Lambda_V$. Again if $V$ has Euler characteristic zero, $\Lambda_V$ is a Legendrian unknot. 
 Another way to see the non-existence of a canonical model is to observe that there is no canonical way to extend the one-dimensional stabilized arc to a higher dimensional Legendrian disk such that the extension is standard near the boundary.

\begin{figure}
    \centering
    \includegraphics[width=9.8 cm]{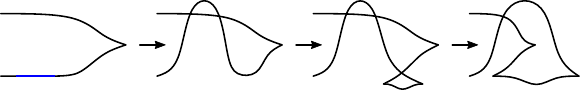}
    \caption{Near the cusp of the black Legendrian, we push through the blue subdomain which perturbs the knot into becoming loose.} 
    \label{fig:push_to_loose}
\end{figure}

This indicates a proliferation of loose Legendrians unknots, making it impossible to speak of \textit{the} loose Legendrian unknot. However, the main result proven in \cite{Murphy_12_LLEHD} about loose Legendrians is that they satisfy an h-principle, which essentially says that, if loose Legendrians $\Lambda_1, \Lambda_2$ in $(Y, \xi)$ are smoothly isotopic (and satisfy additional tangential data), then they are Legendrian isotopic. 
This means that the symplectic geometry of loose Legendrians reduces just to classical differential topology. In particular, all loose Legendrian unknots are Legendrian isotopic. 
For $n = 1$, the local model still makes sense but the h-principle does not hold and therefore, we speak of loose Legendrians only for $n \ge 2$.

Using loose Legendrians, Cieliebak and Eliashberg defined the class of flexible Weinstein domains \cite{Cieliebak_Eliashberg_2012_FSWB}. 

\begin{definition} 
A Weinstein domain  is 
\emph{flexible} if the attaching spheres for its half-dimensional handles are loose Legendrians.
\end{definition}
More generally, we can consider flexible Weinstein sectors. A sector is equivalently a Weinstein domain with the extra data of a Weinstein hypersurface in its contact boundary. A flexible sector is one for which the ambient domain is flexible and the Weinstein hypersurface is loose (in the sense that all core disks of its critical handles are loose Legendrians).
We point out that, in work of Murphy-Siegel~\cite{MurSie18}, these domains are called \textit{explicitly} flexible and Weinstein domains are called flexible if they admit a Weinstein homotopy to an explicitly flexible one; this notion of flexibility is tautologically preserved under Weinstein homotopy.

\subsection{Antisurgery construction of $P$-loose Legendrian unknots}\label{sec: antisurgery construction of $P$-loose legendrians}

In this section, we review the construction of $P$-loose Legendrian unknots via antisurgery in \cite{Lazarev_Sylvan_2023_PLWS}. There, the authors considered Lagrangians disks $D_P \subset T^*D^n$ (described below) with Legendrian boundary $\partial D_P\subset \partial_\infty T^*D^n$ disjoint from the boundary of the zero-section $\partial D^n$. Then, they carved out these disks to obtain subdomains $T^*D^n \setminus D_P$ of $T^*D^n$. 

To construct $P$-loose Legendrian unknot $\Lambda_P$, we view $T^*D^n\setminus D_P$ as a Weinstein sector, or equivalently as a Weinstein domain $B^{2n}\setminus D_P$ with a Legendrian stop in its contact boundary. This can be done as follows.
As a Weinstein domain, $B^{2n} \setminus D_P$ is just $B^{2n} \cup H^{n-1}$.
Since $\partial D^n$ and $\partial D_P$ are disjoint, $\partial D^n$ can be viewed as a Legendrian in the carved out domain $B^{2n} \cup H^{n-1} = B^{2n} \setminus D_P$. 
Next, we recover $B^{2n}$ by attaching a flexible handle along any Legendrian that is loose in the complement of $\partial D^n \subset \partial (B^{2n}\cup H^{n-1})$. The image of $\partial D^n$ in this new $B^{2n}$ is denoted $\Lambda_P$, the $P$-loose unknot. 

To complete the construction of $\Lambda_P$, we need to describe the construction of the Lagrangian disks $D_P \subset T^*D^n$.  These disks were originally introduced by Abouzaid and Seidel in Section 3b of \cite{Abouzaid_Seidel_2010_ASMHR}. Let $U \subset S^{n-1}$ be a compact codimension zero submanifold with smooth boundary. Let 
$$f: S^{n-1} \rightarrow \mathbb{R}$$ 
be a $C^1$-small function with zero as a regular value, so that $f$ is strictly negative in the interior of $U$, zero on $\partial U$, and strictly positive on $S^{n-1} \setminus \overline{U}$. Next, we consider $S^{n-1}$ as the $\frac{1}{2}$-radius sphere $S^{n-1}_\frac{1}{2} \subset D^n$ and extend $f$ to a smooth Morse function (again called)
$$f: D^n \rightarrow \mathbb{R}$$ so that $f$ is $C^0$-small in the $\frac{1}{2}$-radius disk and satisfies 
$$f(tq) = |t|^2 f(q), \quad \text{for} \quad q \in S^{n-1}_\frac{1}{2}, \quad 1\leq t\leq 2.$$ 
Let $\Gamma(df)$ be the graph of $df$ in $T^*D^n$ and let 
$$D_U = \Gamma(df) \cap B^{2n}.$$
Since $f$ is homogeneous for $|q| \ge \frac{1}{2}$ and $0$ is a regular value of $g$, $D_U$ has Legendrian boundary (with respect to the standard radial Liouville vector field on $B^{2n}$) that is disjoint from $\partial D^n$. Furthermore, there is a Lagrangian isotopy $\Gamma(d(sf))_{s \in [0,1]}$ from the zero-section $D^n \subset T^*D^n$ to $D_U$ (that intersects the stop $\partial D^n$ precisely when $s=0$). 

Now, let $U$ be a $P$-Moore space: a CW complex whose reduced singular cohomology is isomorphic to $\mathbb{Z}/P\mathbb{Z}$ in some degree. For example, one can take the mapping cone of the degree $P$ map $$\phi_P: S^k \rightarrow S^k. $$ 
Then the resulting Lagrangian disk $D_U$ is called $D_P$, and is the Lagrangian disk $D_P$ used earlier to construct $\Lambda_P$.

\begin{remark}
Note that the above construction of $D_U \subset T^*D^n$ requires a smooth embedding $U \subset S^{n-1}$. If $U$ is a $P$-Moore space, this is possibly only if $n \ge 5$, hence there exist $P$-loose Legendrians of dimension at least four.
\end{remark}

The sector $T^*D^n \setminus D_U$ is obtained by doing antisurgery on the boundary of $D_U$, while keeping track of the original stop $\partial D^n$ of $T^*D^n$. In particular, the Legendrian boundaries $$\Lambda_U :=  \partial D_U \text{ and } \Lambda^{n-1} :=  \partial D^n $$  
are linked in a way we now describe. Note that we can view $\Lambda^{n-1} \subset \partial T^*D^n$ as the Legendrian unknot in $\mathbb{R}^{2n-1}$.

\begin{proposition} \label{prop: D_U Legendrian}
$\Lambda_U$ is contained in a small neighborhood of $\partial D^n$, and is given by the 1-jet of $f|_{\partial D^n}$, and its front projection in $\partial D^n\times \mathbb{R}$ is given by the graph of $f|_{\partial D^n}$. In particular, $\Lambda_U$ can be isotoped so that its front is obtained from $\partial D^n$ by a negative Reeb pushoff on $U'$, a smaller open set $U' \subset U$, and a positive Reeb pushoff in $\partial D^n \setminus U''$ for some larger open set $U'' \supset U$.  See  Figure~\ref{fig:D_U boundary}.
\end{proposition}

\begin{proof}
As stated in the construction of $D_U = \Gamma(df)$ above, to ensure that $D_U$ has Legendrian boundary, we consider $T^*D^n$ as the Liouville domain 
$$
(B^{2n}, \frac{1}{2}(\sum_{i=1}^n x_i dy_i - y_i dx_i))
$$
equipped with a stop $\partial D^n$, which we identify with the Legendrian unknot. The Liouville vector field associated to the Liouville form $\frac{1}{2}(\sum_{i=1}^n x_i dy_i - y_i dx_i)$ 
is the radial vector field $\frac{1}{2}\sum_{i=1}^n (x_i \partial_{ x_i} + y_i \partial_{y_i})$. So the fact that $f$ has form $t^2 f(\theta)$ near $\partial D^n = S^{n-1}$, with coordinate $\theta$, implies that this Liouville vector field is tangent to $\Gamma(df)$ and hence $\Gamma(df)$ has Legendrian boundary. Near $\partial D^n$, $(B^{2n}, \frac{1}{2}(\sum_{i=1}^n x_i dy_i - y_i dx_i))$ equals 
$$(T^*S^{n-1}, \lambda_{T^*S^{n-1}}) \times (T^*[1-\epsilon, 1], \frac{1}{2}(tdp - p dt)) = (T^*S^{n-1}\times T^*[1-\epsilon, 1], 
\lambda_{T^*S^{n-1}} + \frac{1}{2}(tdp - p dt))
$$
To see this, note that both of these Liouville structures have the radial vector field as their Liouville vector field, which determines the Liouville form. The (convex) contact boundary of this domain is 
$T^*S^{n-1} \times T^*_1 [1-\epsilon, 1]$, the restriction to the cotangent fiber at $1$. The induced contact form is
$$\lambda_{T^*S^{n-1}} + \frac{1}{2}dp$$
given by restricting the Liouville form 
$\lambda_{T^*S^{n-1}} + \frac{1}{2}(tdp - p dt)$ to the cotangent fiber at $1$. Furthermore, this contact boundary is contactomorphic to 
$$(J^1(S^{n-1}) = T^*S^{n-1}\times \mathbb{R}_z, \lambda_{T^*S^{n-1}} + dz)$$
via the map $z = \frac{1}{2}p$. 

Near its boundary, our Lagrangian is given by 
$$
d(t^2 f(\theta)) = t^2 d_\theta f + 2t f(\theta) dt
$$
and hence its Legendrian restriction to the contact boundary $T^*S^{n-1} \times T^*_1 [1-\epsilon, 1]$ is $(d_\theta f, 2 f(\theta))$. Under the contactomorphism $z = \frac{1}{2}p$, this Legendrian maps 
to $(d_\theta f, f(\theta))$, the 1-jet of the function $f(\theta)$. Furthermore, we can isotopy $f$ through functions vanishing precisely on $\partial U$ so that $f$ is equal to $-1$ on $U'$
for a smaller open set $U' \subset U$  and equal to $+1$ on  $S^{n-1}\setminus U''$, for a larger open neighborhood $U'' \supset U$, giving us the claimed result.
\end{proof}
\begin{figure}
    \centering
    \begin{tikzpicture}
    \node at (0,0) {\includegraphics[width=6.9cm]{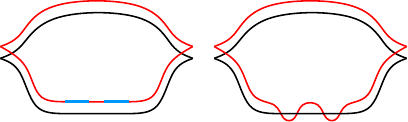}};
    \node[color=cyan] at (-2.1,-0.3) {$U$};
    \node[color=red] at (3,1.2) {$\partial D_U$};
    \node[color=red] at (-3.3,1.2) {$R_\epsilon(\partial D)$};
    \node at (-3,-1.2) {$\partial D$};
    \node at (3,-1.2) {$\partial D$};
    \end{tikzpicture}
    \caption{Left: Two parallel Legendrian unknots. The blue region in the top Legendrian is the smooth subdomain $U$. Right: The Legendrian link $\partial D_U \cup \partial D$, with $\partial D_U$ in red and $\partial D$ in black. This figure is obtained by pushing through $U$ in the top Legendrian down past the bottom Legendrian. } 
    \label{fig:D_U boundary}
\end{figure}
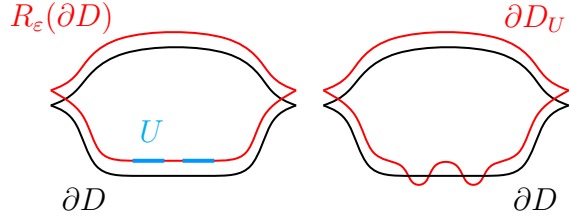

One of our goals is to is to do antisurgery on $\partial D_U$ and find a presentation for $\partial D^n$ in the resulting contact manifold.

\section{Constructions of Higher Dimensional Legendrian Isotopy Moves}\label{sec: Legendrian_moves}
In this section we describe constructions of some higher dimensional Legendrian isotopies. These isotopies are compactly supported, so we may view these as higher dimensional Legendrian moves. In the subsequent section, we will use these Legendrian moves in the construction of handle diagrams for Weinstein manifolds obtained via antisurgery.

\subsection{Suspensions of Legendrian isotopies}
Consider a Darboux chart with coordinates $x_1, \dots, x_n, y_1, \dots, y_n, z$ and contact form $\alpha = dz - \sum_{i = 1}^n y_i dx_i$. Let us use $\R^{2k}$ to refer to the span of $x_1, \dots, x_k, y_1, \dots, y_k$ and $\R^{2(n-k) + 1}$ for the span of $x_{k+1}, \dots, x_n, y_{k+1}, \dots, y_n, z$. Then, with the contact form 
$$\alpha_{n-k} := dz - \sum_{i = k+1}^n y_i dx_i,$$
$(\R^{2(n-k) + 1}, \alpha_{n-k})$ is a contact manifold. With this notation in mind, we may view
\begin{align*}
    \R^{2n + 1} = \R^{2k} \times \R^{2(n-k) + 1} \cong T^* \R^n \times \R^{2(n-k) +1}.
\end{align*}
Consider the Legendrian (with boundary) in this Darboux chart given by 
\begin{align*}
    D^k \times \{0\} \times D^{n-k} \times \{0\} \subset \R^k_{x_1, \dots, x_k} \times \R^k_{y_1, \dots, y_k} \times \R^{n-k}_{x_{k+1}, \dots, x_n} \times \R^{n-k + 1}_{y_{k+1}, \dots, y_n, z}.
\end{align*}
For ease of notation we write $ D^k \times D^{n-k} : =  D^k \times \{0\} \times D^{n-k} \times \{0\}$.
We view $D^k \times D^{n-k}$ as a $D^k$-parameter family of Legendrian $D^{n-k}$'s. Let $s \in D^k$ and $\theta \in D^{n-k}$ denote arbitrary elements.

Consider a Legendrian isotopy
\begin{align*}
    \psi: D^{n-k} \times [0,1]_t \to \R^{2(n-k) + 1}.
\end{align*}
We use the same notation
$\psi: \R^{2(n-k)+1} \times [0,1] \to \R^{2(n-k) + 1} $ to denote a contact isotopy that extends this Legendrian isotopy.
Assume that $\psi$ is identity near $\partial D^{n-k}$ and is $t$-independent near $\p [0,1]$, ie.,
\begin{align*}
    \frac{\p \psi}{\p t}(\theta, 0) = \frac{\p \psi}{\p t}(\theta, 1) = 0 \quad \text{ for all } \theta \in D^{n-k}.
\end{align*}
Fix a smooth ``bump function" $\beta_k: D^k \to [0,1]$ on the parameter space, such that
\begin{itemize}
    \item $\beta_k$ has a unique critical point which is a maxima at $0$ with $\beta_k(0) = 1$, 
    \item $\beta_k$ has radial symmetry, that is, $\beta_k(t) = \beta_k(t')$ whenever $|t| = |t'|$, and
    \item $\beta_k|_{\p D^k} \equiv 0$.
\end{itemize}

\begin{definition}\label{defn: suspension of an isotopy} 
    Define $\Sigma_{D^k}\{\psi\}$, the {\bf $D^k$-suspension of a Legendrian isotopy}, to be the unique Legendrian lift of 
    \begin{align*}
       \left\{ \left(  s, \psi(\theta, \beta_k(s))\right) \quad \big| \quad s \in D^k, \theta \in D^{n-k}  \right\}
    \end{align*}
    under the projection
    \begin{align*}
        \Pi : T^* \R^k \times \R^{2(n-k) + 1} &\to \R^k \times \R^{2(n-k) + 1},\\
        (x_1, \dots, x_k, y_1, \dots, y_k, x_{k+1}, \dots, y_n, z) &\mapsto (x_1, \dots, x_k, x_{k+1}, \dots, y_n, z).
    \end{align*}
\end{definition}
The Legendrian condition
implies that the momentum coordinates $y_1, \dots, y_k$ of $T^* \R^k$,
and thus the Legendrian submanifold itself, can be uniquely recovered from its projection $\Pi(\Sigma_{D^k}\{\psi\})$. 
The boundary of $D^k \times D^{n-k}$ is equal to 
$ \p D^k \times D^{n-k} \cup D^k \times \p D^{n-k}.$
By our assumptions on $\psi$ and $\beta_k$, and uniqueness of the Legendrian lift, 
\begin{align*}
    \p \Sigma_{D^k} \{\psi\} = \p \left(D^k \times D^{n-k}\right).
\end{align*}
We now prove Proposition~\ref{prop: boat_move_intro} from the introduction which states that $\Sigma_{D^k} \{\psi\}$ is Legendrian isotopic to $D^k \times D^{n-k}$ relative to boundary.
\begin{proof}[Proof of Proposition~\ref{prop: boat_move_intro}]
   We can define the isotopy by setting, for every $\tau \in [0,1]$, $\phi_\tau(D^k\times D^{n-k})$ to be the unique Legendrian lift of 
    \begin{align*}
     \left\{ \left(  s, \psi_{\tau \beta_k(s)}(\theta)\right) \quad \big| \quad s \in D^k, \theta \in D^{n-k}  \right\}
    \end{align*}
    under the canonical projection $\Pi$.
    Therefore, $\phi_\tau$ is a Legendrian embedding and $\phi$ is a Legendrian isotopy from $D^n\times D^{n-k}$ to $\Sigma_{D^k}(\psi)$. 

As $\p \Sigma_{D^k} \{\psi\} = \p \left(D^k \times D^{n-k}\right)$, 
for every $\tau \in [0,1]$, $\phi_\tau$ is the identity on the boundary $\p \left(D^k \times D^{n-k}\right)$.
\end{proof}

\begin{remark}
    The word ``suspension" in the context of Legendrians has been used to denote suspensions when the parameter space is $S^k$. In \cite{Dimitroglou_Rizell_Golovko_2021_OLPTS} the suspension $\Sigma_{S^k}\{\Lambda_\theta\}$ is defined similarly to Definition~\ref{defn: suspension of an isotopy} for $S^k$-parameterized families of (closed) Legendrian embeddings. 
    If the $S^k$-family is taken to be a constant family, then one recovers the $S^k$-spun of Legendrians which first appeared in \cite{Ekholm_Etnyre_Sullivan_2005_CHLSR} for $k = 1$ and then was extended to general $k$ in \cite{Golovko_2022_NINELFSS}.

    In this paper, our construction is a ``local" construction. We obtain a Legendrian with boundary within a Darboux chart. This is in contrast to the previous constructions where the obtained Legendrian is closed, that is, has no boundary.
    
    Another small point of difference from earlier constructions is that
    instead of beginning with a $D^k$-parametrized family of Legendrian embeddings, we first create such a family from a 1-parameter Legendrian isotopy.
\end{remark}
We may generalize the suspension construction to the case when the initial Legendrian has a graphical front for a nonzero function. That is, instead of $D^k \times D^{n-k}$, consider Legendrian $N$ contained in a Darboux chart so that the front of $N$ in $D^n\times D^{n-k}\times \R$ is given by 
\begin{align*}
    \pi(N) = \Gamma(f_1 + f_2) = \{(x_1, \dots, x_k, x_{k+1}, \dots, x_n, f_1(x_1, \dots, x_k) + f_2(x_{k+1}, \dots, x_n) \}
\end{align*}
for two smooth functions
\begin{align*}
    f_1: D^k \to \R, \quad f_2: D^{n-k} \to \R.
\end{align*}

Let $\Lambda_{f_2}$ be the unique Legendrian lift in $\R^{2(n-k) + 1}$ of the front $\Gamma(f_2)$.  Again consider a Legendrian isotopy
\begin{align*}
    \psi : \Lambda_{f_2} \times [0,1] \to \R^{2(n-k) + 1}
\end{align*}
which is identity near the boundary of $\Lambda_{f_2}$ and satisfies
\begin{align*}
    \frac{\p \psi}{\p t}(\theta, 0) = \frac{\p \psi}{\p t}(\theta, 1) = 0 \quad \text{ for all } \theta \in  \Lambda_{f_2}.
\end{align*}
\begin{definition}\label{defn:graphical suspension of isotopy}
Define $\Sigma_{\Gamma(f_1)}\{\psi\}$, the {\bf $\Gamma(f_1)$-suspension of a Legendrian isotopy}, to be the unique Legendrian lift of
    \begin{align*}
        \Pi (\Sigma_{\Gamma(f)}\{\psi\}) = \left\{ \left(  s, (0, 0, f_1(s)) + \psi_{\beta_k(s)} (\theta) \right) \in D^k \times \R^{2(n-k)+1} \quad \big| \quad s \in D^k, \theta \in \Lambda_{f_2}  \right\}.
    \end{align*}
\end{definition}

Here, $s \in D^k$ is in the first $k$ position coordinates and $(0,0,f_1(s)) \in \R^{n-k} \times \R^{n-k} \times \R$ and $\psi_{\beta_k(s)}(\theta) \in \R^{2(n-k)+1}$. 
Analogous to Proposition~\ref{prop: boat_move_intro}, we have the following proposition. We do not include a detailed proof as it would merely be a slight tweak of the proof of Proposition~\ref{prop: boat_move_intro}.

\begin{proposition}
    The $\Gamma(f)$-suspension, $\Sigma_{\Gamma(f)}\{\psi\}$, is Legendrian isotopic to $N$ relative  boundary.
\end{proposition}

\subsection{Boat move}
In this section, we introduce a new Legendrian move that can be viewed as a generalization of the Reidemeister $1$ move. Introducing a boat is a way of swapping non-maxima critical points in a graphical front with maximas and some cusp singularities. 

Consider a Legendrian $\Lambda \subset (Y, \xi)$ with front equal to the graph of a Morse function, i.e. the front in some Darboux chart is $\Gamma(f)$ for $f: D^n \to \R$ a Morse function. Suppose there exists a minima. Then, the $n$-dimensional Reidemeister $1$ move swaps in a maxima for the minima, as seen in Figure~\ref{fig:min_to_max}. We want to convert other index critical points to maximas as well.
\begin{definition}\label{defn:boat move}
Suppose the front of an open subset $\Lambda_0 \subset \Lambda$ locally is given as the graph of a Morse function $f: D^n \to \R$ with an index $k$ critical point at $\mathbf{0}$. By the Morse lemma, there exist local coordinates $x_1, \dots, x_n$ such that
\begin{align*}
    f(x_1, \dots, x_n) = -x_1^2 - \dots - x_k^2 + x_{k+1}^2 + \dots + x_n^2.
\end{align*}
Locally we can extend $(x_1, \dots, x_n)$ to Darboux coordinates $(x_1, \dots, x_n, y_1, \dots, y_n, z)$, as any diffeomorphism of $\R^n$ can be extended to a contactomorphism of the $1$-jet space, $J^1(\R^n)$. If we set 
\begin{align*}
    f_1: \R^k \to \R \quad &f_1(x_1, \dots, x_k) = -x_1^2 - \dots - x_k^2, \text{ and }\\
    f_2: \R^{n-k} \to \R \quad &f_2(x_{k+1}, \dots, x_n) = x_{k+1}^2 + \dots + x_n^2,
\end{align*}
we are in the set up of Definition~\ref{defn:graphical suspension of isotopy}. 
We define the {\bf $(n,k)$-boat move} to be the replacement of the graphical open subset $\Lambda_0$ by the $\Gamma(f_1)$-suspension of $R1_{n-k}$, $\Sigma_{\Gamma(f_1)} (R1_{n-k})$. Let $B_{n,k}$ denote the resulting Legendrian (locally). We refer to $B_{n,k}$ as the {\bf $(n,k)$-boat}.
\end{definition}
\begin{figure} 
    \centering
    \includegraphics[width=12 cm]{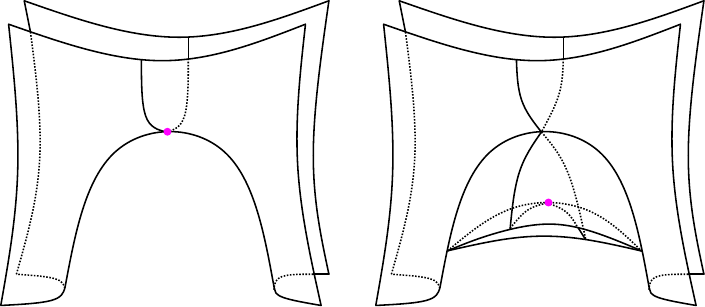}
    \caption{A (2,1)-boat move which changes a saddle point into a maximum.}
    \label{fig:surface_boat_move}
\end{figure}

\begin{remark}
    Note that if $k=n$, the boat move does not change anything. If $k = 0$, the boat move is equal to the $n$-dimensional first Reidemeister move.
\end{remark}

\begin{proposition}
    The $(n,k)$-boat $B_{n,k}$ is Legendrian isotopic to $\Lambda_0$ relative boundary.
\end{proposition}

\begin{proof}
     As the boat move is a special case of the suspension defined in Definition~\ref{defn:graphical suspension of isotopy}, $B_{n,k}$ is Legendrian isotopic to $\Lambda_0$ relative to boundary directly from Proposition~\ref{prop: boat_move_intro}.
\end{proof}

\begin{remark}
    The boat move gets its name because the $(2,1)$-boat looks like a overturned boat or canoe as seen in Figure~\ref{fig:surface_boat_move}. The $(2,1)$-boat is very similar to the uni-germ $A^{e, \pm}_3$, birth of cuspidal lips, see \cite{Arnold_1990_SCWF} or \cite{Goryunov_Alsaeed_2015_LIFF3M}. The main difference is the parametrization of the Legendrian front before the move.
\end{remark}

Our goal for introducing the boat was that we wanted all the Reeb chords to correspond to maximas of the front. Of course, the new front has non-smooth points.
\begin{proposition}\label{prop:boat_move_makes_max}
   Decompose the front of $B_{n,k}$ into a disjoint union of a finite number of graphical components away from non-smooth points. All but one of these graphical components are graphs of functions with no critical points. Further, the unique component with critical points has a unique critical point that is a maxima. Moreover, the tangent plane at the cusps are not parallel to $(x_1, \dots, x_n)$-plane.
\end{proposition} 

\begin{proof}
The front of $B_{n,k}$ has a critical point only at points where both the front of $R1_{n-k}$ and $f_1$ have critical points. This happens at only one point, namely, $x_1 = \dots = x_n = 0$. By construction of $f_1$, $x_1 = \dots = x_n =0$ is assumed to be a maxima in the first $k$ coordinates. The Reidemeister move then converts the critical point to a maxima in the last $n-k$ coordinates, thus proving the proposition.

    We get that the cusps are not parallel to the $(x_1, \dots, x_n)$-plane as the $(n-k)$-dimensional Reidemeister front does not have cusps parallel to the $(x_{k+1}, \dots, x_n)$-plane. 
\end{proof}

\begin{remark} \label{DGAremark}
Consider a Reeb chord with one end on the pink dot in the left figure of Figure~\ref{fig:surface_boat_move}. This Reeb chord survives the boat move. One may be confused about how the grading of this Reeb chord as an element of the Legendrian dga changes, since the local Morse index is changing. However, note that the relationship between the grading and local Morse index is not the typical one 
\cite{Ekh_Etn_Sul_non_isotopic}, as the cusps we use do not rotate the Lagrangian planes the way typical cusps in the literature do. As such, one can check that the associated Maslov indices of loops of tangent vectors from the top to the bottom of the Reeb chord actually stay the same under our boat moves, preserving the grading.
\end{remark}

\section{Weinstein presentations of subdomains}\label{sec: main construction}

In this section, we prove Theorem~\ref{thm: general_antisurgery_intro} and Corollary~\ref{cor: explicit_P_loose_intro}. We then apply this theorem to construct several explicit examples of handle decompositions for Weinstein manifolds obtained via antisurgery. We conclude with some open questions.

\subsection{Proof of Main Theorem \ref{thm: general_antisurgery_intro} }

Recall that Theorem \ref{thm: general_antisurgery_intro} required that certain Reeb chords be non-degenerate and correspond to critical points of Morse functions. The following result shows that this condition can always be achieved. 
\begin{lemma}\label{lemma:non-degenerate_Reeb_chords}
    Let $\Lambda_1$ and $\Lambda_2$ be two Legendrians of a contact manifold $M$ which are $C^0$-close to one another; hence we can assume $\Lambda_2$ in the 1-jet space $J^1(\Lambda_1)$ of $\Lambda_1$.     
    Then, we may perturb $\Lambda_2$ by a $C^0$-small isotopy so that there are finitely many Reeb chords from $\Lambda_1$ to $\Lambda_2$ 
    and so that in a neighborhood of each Reeb chord endpoint on $\Lambda_2$, $\Lambda_2$ looks like $(x,df,f)$ for $f$ a local Morse function on $\Lambda_1$.
\end{lemma}

\begin{proof}
By definition, $J^{1}(\Lambda_1) = T^{*}\Lambda_1 \times \mathbb{R}$ and Reeb chords between $\Lambda_1, \Lambda_2$ project to double points under the Lagrangian projection $\Pi: J^1 (\Lambda_1) \to T^{*}(\Lambda_1)$. We may perturb $\Lambda_2$ by a $C^{\infty}$-small Legendrian isotopy so that these double points are transverse in $T^*\Lambda_1$. Let $q \in \Lambda_1 \subset T^*\Lambda_1$ denote one of these double points. Then we may further perturb $\Lambda_2$ so that $\Pi(\Lambda_2)$ is additionally transverse to $T^*_q \Lambda_1$ at $q$ (in addition to being transverse to $\Lambda_1$), without creating any new double points. Since $\Pi(\Lambda_2)$ is transverse to $T^*_q \Lambda_1$, we have that $\Pi(\Lambda_2)$ is given by $(x_1, \cdots, x_n, \partial_1 f(x), \cdots, \partial_n f(x))$ for some function $f$ with critical point at $q$. Furthermore, since the projection $\Pi(\Lambda_2)$ is transverse to $\Lambda_1$, this critical point is non-degenerate and hence is a Morse critical point. In particular, $q$ corresponds to a Reeb chord between $\Lambda_1, \Lambda_2$ with the desired property. See Figure \ref{fig:not_at_cusp} for an example of this perturbation.

\end{proof}

\begin{figure}
    \centering
    \includegraphics[width=5.95 cm]{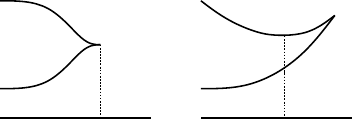}
    \caption{Front projections of two Legendrian $\Lambda_1$ (zero-section) and $\Lambda_2$ (with the cusp).   Left figure: there is a Reeb chord (depicted by the dotted line) between the Legendrian $\Lambda_2$ and $\Lambda_1$ with endpoint on the cusp of $\Lambda_2$. Right figure: after Legendrian isotopy of $\Lambda_2$, the Reeb chord between $\Lambda_1, \Lambda_2$ has endpoint on a smooth branch of $\Lambda_2$ which is locally described by the graph of a Morse function on $\Lambda_1$ (with minimum corresponding to the Reeb chord endpoint).   
}
    \label{fig:not_at_cusp}
\end{figure}

\begin{remark} \label{rem: maximas are minimas}
    In the following Lemma \ref{lem:one surgery on lambda f is enough}, the Reeb chords we investigate may be interpreted as lying between $\Lambda$ and $\Lambda_-$ or between $\Lambda$ and $\Lambda_+$. This is relevant when we refer to the Reeb chords as being represented as maxima or minima of a Morse function in a suitable neighborhood; what is maxima with regard to a function on $\Lambda_-$ will be a minima with regard to a function on $\Lambda_+$. Here, it should be understood that we can only handleslide $\Lambda$ past $\Lambda_-$ at Reeb chords represented by maxima with respect to a locally-defined Morse function on $\Lambda_+$ (alternatively mimima with respect to a locally-defined Morse function on $\Lambda_-$). The following lemma uses suitable boat moves to turn all critical points into such maxima (resp. minima). 
\end{remark}

\begin{lemma}\label{lem:one surgery on lambda f is enough}	
Suppose $\Lambda$ is a Legendrian submanifold in a contact manifold $(M_0,\xi_0)$, and that $(-1)$-surgery on $\Lambda$ produces contact manifold $(M',\xi')$. Suppose $\Lambda_+$ and $\Lambda_-$ are a pair of $n$-dimensional Legendrian submanifolds which are completely parallel (so that one is an Reeb $\epsilon$-pushoff of the other) in $(M_0,\xi_0)$ but not necessarily parallel in $(M',\xi')$, ie. they may be distinctly linked with $\Lambda$. 

Let $(M,\xi)$ be the $(2n+1)$-dimensional contact manifold obtained from $(M_0, \xi_0)$ by $(-1)$-surgeries along $\Lambda$ and $\Lambda_-$, and a $(+1)$-surgery along $\Lambda_+$. Then, there exists a Legendrian submanifold $\Lambda_f \subset (M_0, \xi_0)$ such that $(M,\xi)$ can be obtained by only $(-1)$-surgery along the components $\Lambda_f$. 
\end{lemma}

\begin{proof}
	To obtain $\Lambda_f$, our goal will be to cancel $\Lambda_{+}$ with $\Lambda_{-}$ in a surgery diagram of $(M,\xi)$. In order to do so, we need $\Lambda_{+}$ and $\Lambda_{-}$ to be completely parallel in $(M', \xi')$, ie. they must be identically linked with $\Lambda$ in $(M_0, \xi_0)$. We will parallelize $\Lambda_{+}$ and $\Lambda_{-}$ by performing a sequence of Legendrian isotopies that preserve the resulting surgered contact manifold $(M, \xi)$, see the example of Figure~\ref{fig:recipe_slides}.

\begin{figure} [h!]
    \centering
    \begin{tikzpicture}
    \node at (0,0) {\includegraphics[width=10.6 cm]{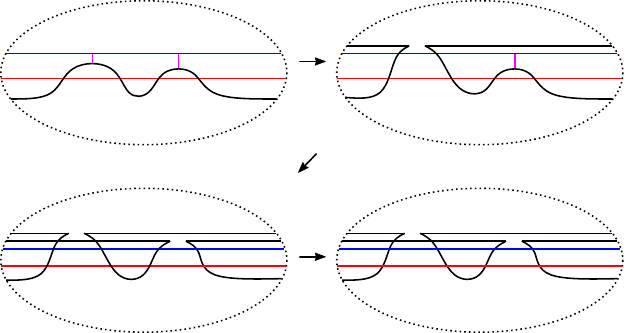}};
    \node[color=blue] at (-3,2.3) {$\Lambda_-$};
    \node[color=red] at (-3.4,1.1) {$\Lambda_+$};
    \node at (-4.3,1) {$\Lambda$};
    \node at (3,-0.8) {$\Lambda_f$};
    \end{tikzpicture}    
    \caption{Assuming they are parallel outside of the dotted circle, the red Legendrian $\Lambda_+$ and blue Legendrian $\Lambda_-$, where we perform a $(+1)$ and $(-1)$ Legendrian surgery respectively, are made parallel by handlesliding the black Legendrian $\Lambda$ over $\Lambda_-$ at its maxima. The $\Lambda_+$ and $\Lambda_-$ are then cancelled, leaving only the new black Legendrian $\Lambda_f$.} 
    \label{fig:recipe_slides}
\end{figure}
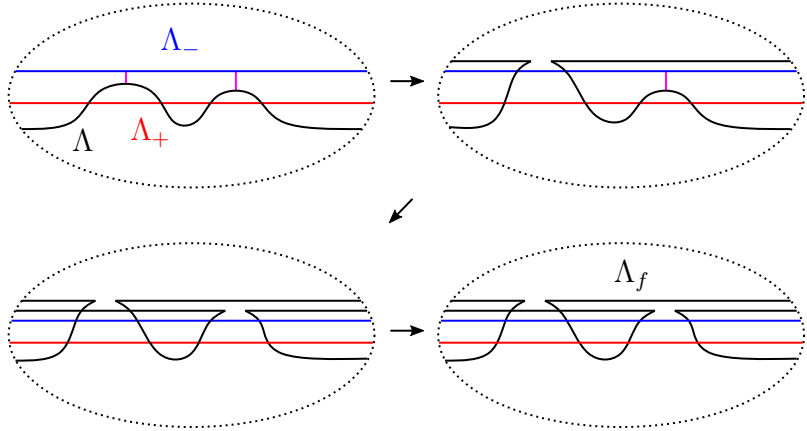

\begin{figure} [h!]
    \centering
    \begin{tikzpicture}
    \node at (0,0) {\includegraphics[width=10.45 cm]{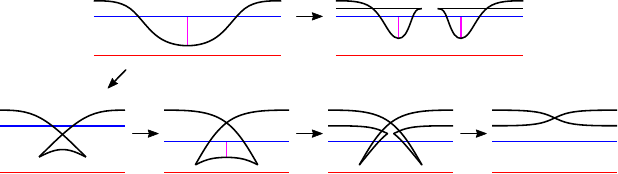}};
    \node[color=blue] at (-4,1.2) {$\Lambda_-$};
    \node[color=red] at (-4,0.6) {$\Lambda_+$};
    \node at (-3.6,1.8) {$\Lambda$};
    \end{tikzpicture}    
    \caption{Top row: doing $-1$ handleslide of $\Lambda$ over $\Lambda_-$ at an index $1$ critical point (of the height difference function) creates two more Reeb chords. Bottom row: first doing a boat move at the index $1$ critical point (a first Reidemeister move) and then a $-1$ handleslide over $\Lambda_-$ removes all Reeb chords. We observe that the resulting Legendrian is a crossing connect sum of $\Lambda$ and $\Lambda_-$ which appear when doing $+1$ handleslide, see Figure~\ref{fig:handleslides}.}     
    \label{fig:no_new_Reeb}
\end{figure}
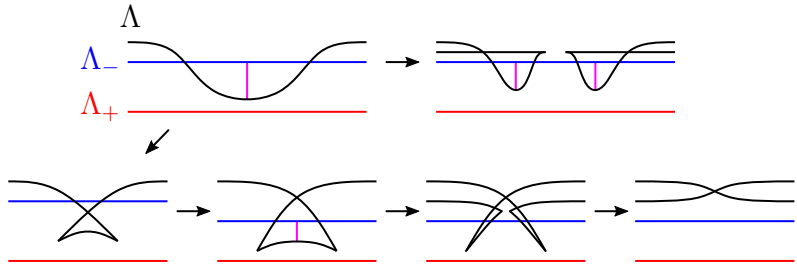

	By Lemma~\ref{lemma:non-degenerate_Reeb_chords}, we may perturb $\Lambda$ by a Legendrian isotopy so that we have finitely many Reeb chords, and in a neighborhood of each Reeb chord, the height difference is Morse function.
	
	If we attempt to isotope $\Lambda_{-}$ towards $\Lambda_{+}$ by pushing $\Lambda_{-}$ in the $z$-coordinate in the front, we are obstructed whenever $f$ has a Morse critical point on a part of $\Lambda$ between $\Lambda_{-}$ and $\Lambda_{+}$. These critical points correspond to non-degenerate Reeb chords between $\Lambda_-$ and $\Lambda$ and must be removed. We do so one at a time: we always work with the smallest remaining critical value, that is, the shortest remaining Reeb chord. 

	Consider the critical point, say $q$, with least critical value. If $q$ has Morse index 0, then in the front projection, we see a local maximum (see Remark \ref{rem: maximas are minimas}). We remove it by performing a handleslide of $\Lambda$ over $\Lambda_{-}$ along the Reeb chord between $\Lambda_-$ and $\Lambda$ corresponding to $q$. We then isotope $\Lambda_{-}$ towards $\Lambda_{+}$. See Figure~\ref{fig:recipe_slides}. We note that no new Reeb chords are created in the process. 
	
	\begin{figure} [h!]
		\centering
  \begin{tikzpicture}
    \node at (0,0) {\includegraphics[width=14.4 cm]{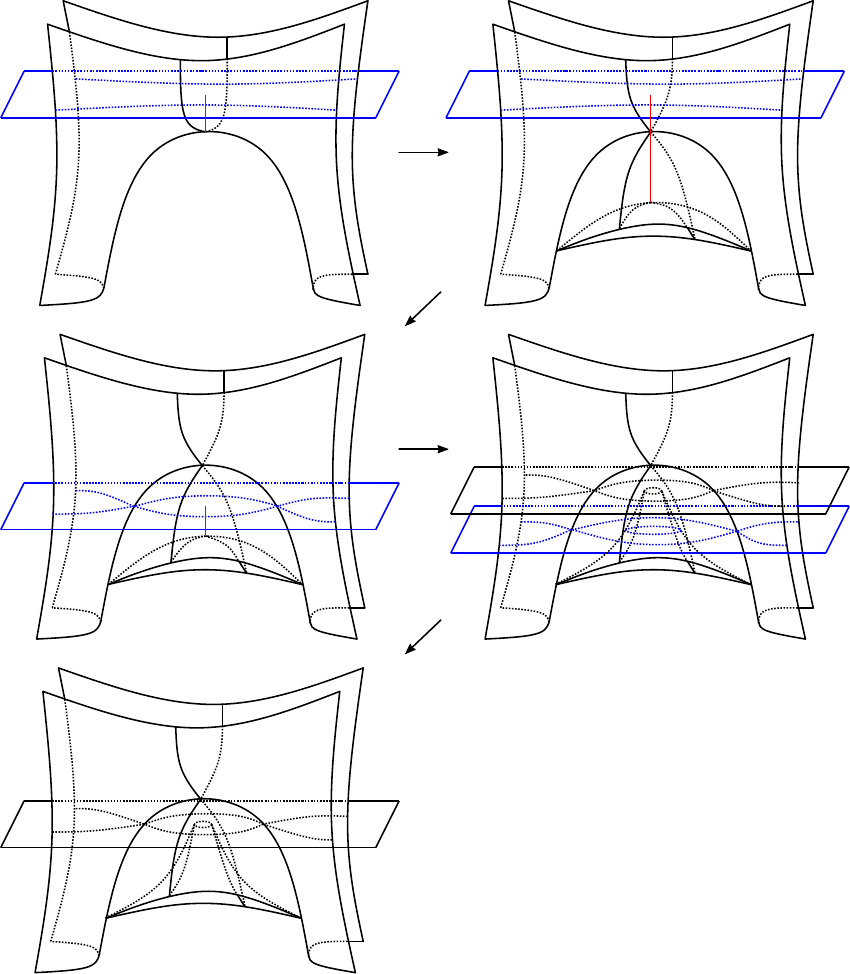}};
    \node[color=blue] at (-6.8,5.8) {$\Lambda_-$};
    \node at (-4.5,5) {$\Lambda$};
    \end{tikzpicture} 
		
		\caption{The process of performing a boat move locally for an index 1 critical point on a surface Legendrian to enable a handleslide:\\\hspace{\textwidth}
			\textbf{(1.)} The black Legendrian $\Lambda$ has an index 1 critical point at which there is a red Reeb chord to the blue Legendrian $\Lambda_-$. \\\hspace{\textwidth}
			\textbf{(2.)} Perform a boat move at the critical point, the Reeb chord now meets the $\Lambda$ at a maximum in the front diagram.\\\hspace{\textwidth}
			\textbf{(3.)} Isotope $\Lambda_-$ downwards, the Reeb chord is now uninterrupted.\\\hspace{\textwidth}
			\textbf{(4.)} Perform a handleslide along the Reeb chord.\\\hspace{\textwidth}
			\textbf{(5.)} Isotope the $\Lambda_-$ past the boat.\\\hspace{\textwidth}} 
		\label{fig:boat_handleslide}
	\end{figure}
	
	If $q$ has Morse index $n-k$, for $0 \leq k < n$, we first perform a $(n, k)$-boat move to $\Lambda$, as defined in Definition~\ref{defn:boat move}. (If the critical point has Morse index $n$, then it is a minimum in the front projection, so we  perform an $(n, 0)$-boat move, which is just an $n$-dimensional first Reidemeister move). By Proposition~\ref{prop:boat_move_makes_max}, the $(n,k)$-boat move is a Legendrian isotopy that converts the index $n-k$ critical point $q$ to a maximum, say $q'$, in the front projection. We now have a Reeb chord, say $\gamma_{q'}$, between $\Lambda$ and $\Lambda_{-}$ corresponding to $q'$. Note that Proposition~\ref{prop:boat_move_makes_max} also implies no other Reeb chords are created in this process. We isotope $\Lambda_{-}$ towards $q'$ until $\gamma_{q'}$ is uninterrupted. Next, we handle slide $\Lambda$ over $\Lambda_{-}$ along the Reeb chord $\gamma_{q'}$ at $q'$. We can then isotope $\Lambda_{-}$ further towards $\Lambda_{+}$. This process is fully illustrated in Figure~\ref{fig:no_new_Reeb} and Figure~\ref{fig:boat_handleslide}. 
 
	We repeat this process until no obstructing non-degenerate Reeb chords remain and $\Lambda_{+}$ and $\Lambda_{-}$ are in cancelling position. Once cancelled, we are left with a single Legendrian $\Lambda_f$ in the surgery diagram of $(M,\xi)$. In summary, $\Lambda_f$ corresponds to $\Lambda$ in the following way: for each Reeb chord $\gamma$ from $\Lambda$ to $\Lambda_-$ with local index $n-k$, we applied an $(n, k)$-boat move to $\Lambda$ and did a cusp connected sum with $\Lambda_-$.
\end{proof}

 \begin{remark}
 We note that if we handleslide over a Reeb chord corresponding to an index $k$ critical point without first doing a boat move, then new Reeb chords are created; see the top row of Figure~\ref{fig:no_new_Reeb}. By first doing the boat move, which does not create any new Reeb chords but only changes the local index of the existing Reeb chord, we ensure that handlesliding removes that Reeb chord without creating any new chords. 
 \end{remark}

\begin{remark}
Although Lemma \ref{lem:one surgery on lambda f is enough} is stated in the language of contact manifolds and contact $(\pm 1)$-surgeries, the result also holds
with Weinstein homotopies. We recall that $(-1)$-surgeries correspond to Weinstein handle attachment while $(+1)$-surgeries correspond to handle removal, and in general do not produce fillable contact manifolds. However, in our case, all surgeries moves are handle-slides are over the $(-1)$-Legendrian and hence we can view the $(+1)$-Legendrian as a placeholder for the boundary of the Lagrangian disk, which will be removed at the last step. Hence all our contact moves are really Weinstein homotopies. 

\end{remark}

\begin{remark}
    In the next two proofs, we abuse terminology and refer interchangeably to Weinstein diagrams and Weinstein domains. So, we sometimes say ``attach a handle to the Weinstein diagram". We will also refer to the carving out of a Lagrangian disk as   ``antisurgery" along its boundary Legendrian. 
    We hope this abuse of terminology will not confuse the reader but, in fact, make the proof easier to read. 
\end{remark}

\begin{proof}[Proof of Theorem~\ref{thm: general_antisurgery_intro}]\label{prf: main construction}
Recall that $L \subset X$ are of the form $X = T^*D^n \cup H_i^n$ and $L = D^n \subset T^*D^n$. Then Weinstein sector $X \setminus L$ is obtained by antisurgery along the Legendrian knot $\partial L$ which corresponds to the unknot $\partial D^n \subset T^*D^n$. Now onward, we will denote the knot $\partial L$ by $\Lambda_+$, and depict it in red in all figures. Let $\Lambda=\cup_i\Lambda_i$ denote the link consisting of all the attaching spheres of the $H^n_i$. 

	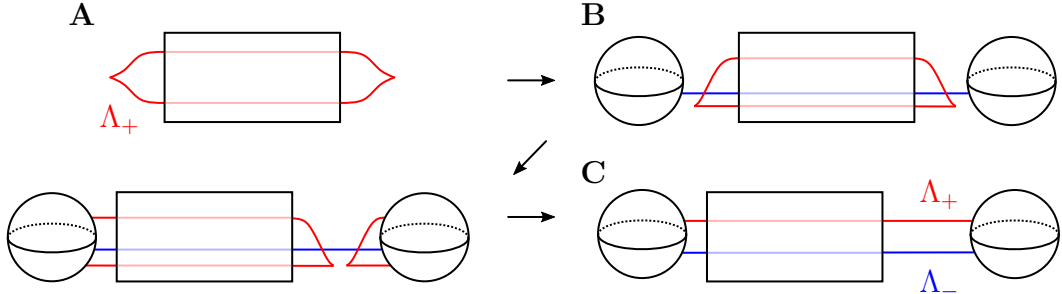
\begin{figure}[h]
		\centering
  \begin{tikzpicture}
    \node at (0,0) {\includegraphics[width=14 cm]{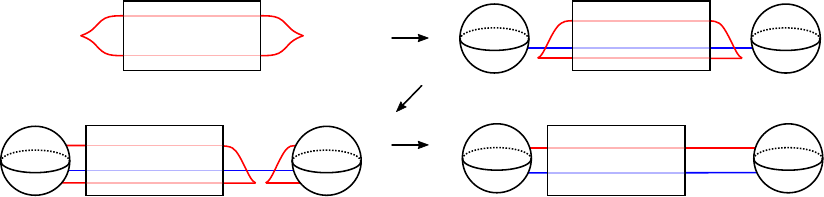}};
    \node[color=red] at (-5.5,0.5) {$\Lambda_+$};
    \node[color=red] at (5.4,-0.5) {$\Lambda_+$};
    \node[color=blue] at (5.4,-1.7) {$\Lambda_-$};
    \node at (-6,1.9) {\textbf{A}};
    \node at (0.8,1.9) {\textbf{B}};
    \node at (0.8,-0.2) {\textbf{C}};
    \end{tikzpicture} 
		\caption{A cartoon depiction of Diagrams A, B, and C. In the top left, we see a red Legendrian unknot in the boundary of a Weinstein domain. This unknot may be linked with attaching handles of the domain away from its cusps, represented by the black rectangle. This diagram is related to the figure on the bottom right by the addition of a cancelling pair and a single handleslide. }
		\label{fig:unknot_antisurgery}
	\end{figure}

	We may isotope $\Lambda_+$ so that the front projection of $\Lambda$ is disjoint from the cusps of $\Lambda_+$, giving us a surgery diagram, say ``Diagram A" for $\partial(X \setminus L)$, as depicted in the top left of Figure~\ref{fig:unknot_antisurgery}. We introduce a cancelling pair to Diagram A---an $(n-1)$-handle and a critical handle $\Lambda_-$ so that $\Lambda_-$ is a parallel pushoff of the bottom arc of $\Lambda_+$---to obtain a new surgery diagram, say ``Diagram B", again for $\partial(X \setminus L)$, as depicted in top right of Figure~\ref{fig:unknot_antisurgery}. We observe that Diagram B is in turn equivalent to a ``Diagram C" of the form of the bottom right of Figure~\ref{fig:unknot_antisurgery}  where $\Lambda_+$ and $\Lambda_-$ both traverse across the $n-1$ handle exactly once each.
	
	We note that, in Diagram C,  $\Lambda_+$ and $\Lambda_-$ are parallel except in how they are linked with $\Lambda$. So, we can apply Lemma~\ref{lem:one surgery on lambda f is enough} to obtain a surgery presentation that contains only $(-1)$-surgery along Legendrian submanifolds, say ``Diagram D". Then, this Diagram D corresponds to a Weinstein handle diagram where each $(-1)$-surgery corresponds to a critical handle attachment.
	
	Note that the Weinstein diagram Diagram D has exactly one more index $n-1$ handle than the Weinstein presentation for $X$. Further, if we denote the attaching spheres of the index $n$ handles $H_i$ in Diagram D by $\Lambda_i'$, resp., the construction in Lemma~\ref{lem:one surgery on lambda f is enough} implies that the $\Lambda_i'$ are exactly as described in the theorem statement.
\end{proof}

As a corollary to this theorem, we explain how to describe the $P$-loose Legendrian unknot in $\mathbb{R}^{2n-1} \subset S^{2n-1} = \partial B^{2n}$. To do so, we first describe a front diagram of the knot in $S^{n-1} \times \mathbb{R}^{n} \subset \partial(B^{2n} \cup H^{n-1})$ and then attach a flexible handle to make the ambient space $\partial B^{2n}$.

\begin{proof}[Proof of Corollary~\ref{cor: explicit_P_loose_intro}]

Recall from Section~\ref{sec: antisurgery construction of $P$-loose legendrians} and \cite{Lazarev_Sylvan_2023_PLWS} that the $P$-loose Legendrian unknot, 
$$
\Lambda_P \subset S^{2n-1} = \partial B^{2n}
$$
is obtained as follows. First, to construct a $P$-Moore space, consider the CW complex $S^1 \cup_P D^2$, the result of attaching $D^2$ to $S^1$ along the degree $p$ map $\partial D^2 = S^1 \rightarrow S^1$. If $n \ge 5$, then this CW complex embeds into $S^{n-1}$, as observed in \cite{Abouzaid_Seidel_2010_ASMHR}, and we let $U$ be a neighborhood of this CW complex. Next,  we observe that $U$ has a Morse function (with gradient outward pointing near the boundary of $U$) that has three critical points, one of index 0 and 1 for the $S^1$ and one of index 2 for the $D^2$.

Next, we carve out a disk $D_U \subset T^*D^n$ from $T^* D^n$ to obtain $T^*D^n \setminus D_U$ which as an unstopped Weinstein domain and is equivalent to $B^{2n} \cup H^{n-1}$. Then, 
$\Lambda_P$ is $$\partial D\subset\partial (B^{2n} \cup H^{n-1}\cup H_{flex}),$$
where $H_{flex}$ is a flexible Weinstein handle attached along a Legendrian which is loose in the complement of $\partial D \subset B^{2n} \cup H^{n-1}$. Additionally, $H_{flex}$ is in cancelling position with $H^{n-1}$.  We will explicitly construct the Weinstein handle decomposition of this $B^{2n} \cup H^{n-1}$ to obtain an explicit front diagram for $\Lambda_P$.

We begin by using the construction from Proposition~\ref{prop: D_U Legendrian}. In the boundary $S^{2n-1} = \partial B^{2n}$, we consider a pair of cancelling contact surgeries along parallel $(n-1)$-dimensional Legendrian unknots. We label the Legendrian unknot corresponding to the $(-1)$-surgery $\Lambda$ and the Legendrian unknot corresponding to the $(+1)$-surgery $\Lambda_+$. $\Lambda_+$ and $\Lambda$ represent Legendrian boundaries $\partial D_U$ and $\partial D$, respectively. Note that there exists a $P$-Moore space $U \subset \Lambda_+$ by our assumption. We perturb $\Lambda_+$ past the $\Lambda$ by ``pushing" the subdomain $U$ so that $\Lambda_+$ and $\Lambda$ are no longer parallel. Let us refer to this surgery diagram by ``Diagram E".

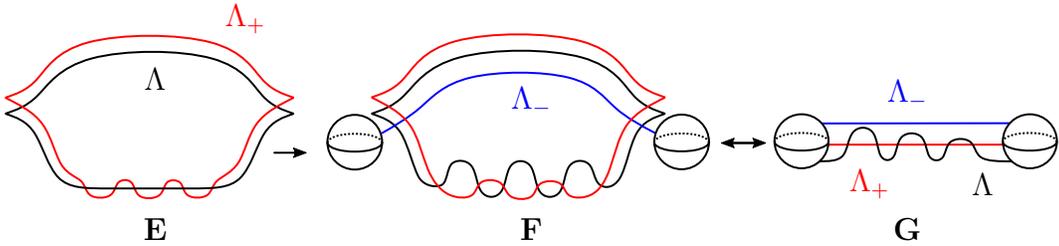
\begin{figure} [h!]
	\centering
 \begin{tikzpicture}
    \node at (0,0) {\includegraphics[width=14cm]{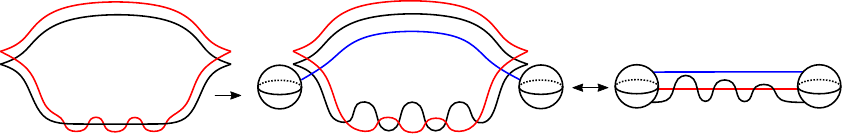}};
    \node[color=red] at (-3.8,1.3) {$\Lambda_+$};
    \node at (-5,0.5) {$\Lambda$};
    \node[color=red] at (4.5,-0.9) {$\Lambda_+$};
    \node[color=blue] at (5, 0.3) {$\Lambda_-$};
    \node[color=blue] at (0, 0.2) {$\Lambda_-$};
    \node at (-5,-1.5) {\textbf{E}};
    \node at (0,-1.5) {\textbf{F}};
    \node at (5,-1.5) {\textbf{G}};    
    \node at (6,-0.9) {$\Lambda$};
    \end{tikzpicture} 
	\caption{A cancelling pair of handles, represented by a pair of spheres and a blue Legendrian, is added just below the cusps of the linked red and black Legendrians. This diagram is equivalent to the final one via two handleslides.} 
	\label{fig:insert_cancelling_handles}
\end{figure}

\begin{figure} [h!]
	\centering
 \begin{tikzpicture}
    \node at (0,0) {\includegraphics[width=8.5cm]{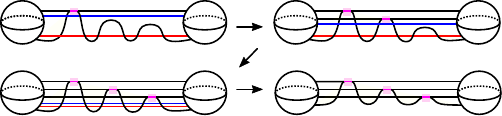}};
    \node at (2.3,-1.2) {\textbf{H}};    
    \end{tikzpicture} 
	\caption{Applying Lemma~\ref{lem:one surgery on lambda f is enough} to Diagram G involves three instances of a boat move isotopy followed by a handleslide. Then $\Lambda_+$ and $\Lambda_-$ are completely parallel, thus the respective surgeries on these knots cancel.} 
	\label{fig:3_boats}
\end{figure}

To Diagram E, we insert a cancelling pair of handles consisting of an $(n-1)$-handle, $H^{n-1}$ and an $n$-handle, $H^n$, just below the cusps of $\Lambda_{+}$ and $\Lambda$, see Figure~\ref{fig:insert_cancelling_handles}. Note that as a contact surgery curve on the boundary, the $n$-handle attachment corresponds to a $(-1)$-surgery along a Legendrian knot. Let us denote the Legendrian knot corresponding to this $(-1)$-surgery by $\Lambda_{-}$, and refer to this Weinstein diagram as ``Diagram F". We see from Figure~\ref{fig:insert_cancelling_handles} that Diagram F is equivalent to ``Diagram G" after some handleslides and Legendrian isotopies. 

We are now set up exactly as in the proof of Theorem~\ref{thm: general_antisurgery_intro}, namely Diagram G is of the form of Diagram C in the proof of Theorem~\ref{thm: general_antisurgery_intro}. So, we similarly apply Lemma~\ref{lem:one surgery on lambda f is enough} to obtain the required Weinstein diagram, say ``Diagram H". Since the $P$-Moore space has three critical points, $\Lambda$ in Diagram G also has three critical points, of index 0, 1, and 2 when considering the Morse function $f$ of Lemma 4.3. Thus in applying the Lemma, we perform three instances of a boat move and a cusp connect sum to $\Lambda$ in Diagram G to obtain Diagram H, see Figure~\ref{fig:3_boats}.

\begin{figure} [h!]
	\centering
  \begin{tikzpicture}
    \node at (0,0) {\includegraphics[width=12.4 cm]{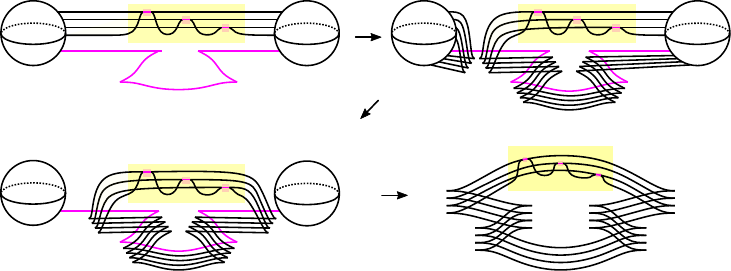}};
    \definecolor{mycolor}{RGB}{255,0,255}
    \node[color=mycolor] at (-4.7,0.9) {$\Lambda_{flex}$};
    \node at (-6.5,2.3) {\textbf{I}};
    \node at (-4.7,2.4) {$\Lambda$};
    \node at (5.3,-1.8) {$\Lambda_P$};
    \node at (1.8,-0.4) {\textbf{J}};
    \end{tikzpicture} 
	\caption{The final step in the construction: in the complement of the black Legendrian $\Lambda$ which traverses the $(n-1)$-handle some number of times, we attach a flexible handle along the pink loose Legendrian $\Lambda_{flex}$. We then repeatedly slide $\Lambda$ over $\Lambda_{flex}$ to detach it from the $(n-1)$-handle. Finally, we cancel $\Lambda_{flex}$ with the $n-1$-handle.} 
	\label{fig:flexible_step}
\end{figure}
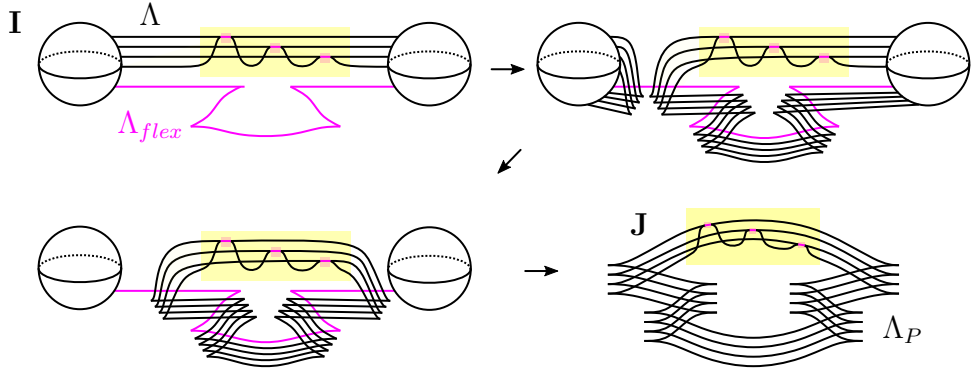

Finally, to obtain $\Lambda_P$ in $B^{2n}$, that is, to make the ambient manifold $B^{2n}$, we attach a flexible handle, $H_{flex}$, to cancel out $H^{n-1}$ in Diagram H. This amounts to attaching a loose Legendrian $\Lambda_{flex}$ that winds around the $(n-1)$-handle once, and is in the complement of $\Lambda$ in Diagram H. This gives us ``Diagram I" which is depicted in the top left of Figure~\ref{fig:flexible_step}, with $\Lambda_{flex}$ denoted in pink.

Next, in Diagram I, we slide $\Lambda$ repeatedly over $H_{flex}$ until $\Lambda$ no longer passes through the $H^{n-1}$. We then cancel $H^{n-1}$ with $H_{flex}$. We are left with a surgery diagram, say ``Diagram J", that consists of a single Legendrian $\Lambda_P$ in $B^{2n}$. This process is illustrated in Figure~\ref{fig:flexible_step}. We see that $\Lambda_P$ consists of four loose Legendrian unknots which are completely parallel away from a bounded region where they are linked (in a way that depends on p) and are connected via three boat moves and cusp connected sum gluings.

As a Weinstein diagram, Diagram J depicts the Weinstein sector $(B^{2n}, \Lambda_P)$, where $\Lambda_P$ is a $P$-loose Legendrian unknot.
\end{proof}

\subsection{Explicit examples} 
We now construct several exotic Weinstein manifolds as applications of the construction from Theorem~\ref{thm: general_antisurgery_intro} and Corollary~\ref{cor: explicit_P_loose_intro}. In all these examples, we consider antisurgery on the Lagrangian disk obtained by perturbing the boundary $S^{n-1} = \partial D^n$ of the zero section $D^n \subset T^*D^n$ in a neighbourhood $U \subset S^{n-1}$, as in the construction of $P$-loose Legendrians (see Section~\ref{sec: antisurgery construction of $P$-loose legendrians}).

\begin{examples} 
Suppose $U= D^{n-1} \subset S^{n-1}$ is a  disk. Then, we may choose the Morse function $g : U \to \R$ to have a single critical point of index 0. So, we obtain $\Lambda_{U}$ with a single maximum. Applying the construction from Theorem~\ref{thm: general_antisurgery_intro}, we obtain a Legendrian  $\Lambda_{f}$ by a single handleslide (see Figure~\ref{fig:example_disk}). The resulting Legendrian, $\Lambda_f$, is a standard Legendrian unknot in the complement of the $(n-1)$-handle, $H^{n-1}$.

\begin{figure} [h!]
    \centering
    \begin{tikzpicture}
    \node at (0,0) {\includegraphics[width=8.95 cm]{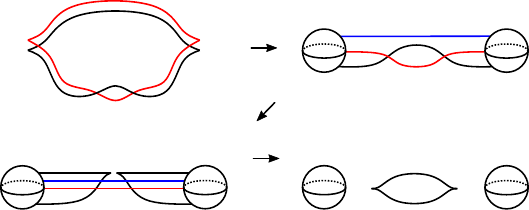}};
    \node[color=red] at (-3.8,1.8) {$+1$};
    \node at (-3.8,0.2) {$\Lambda$};
    \node at (2,-0.9) {$\Lambda_f$};
    \end{tikzpicture}
    \caption{Our construction applied to the case where $U=D^k$: We first add a cancelling pair and slide the red and black Legendrians over the blue Legendrian. Then at the maximum on the black Legendrian, we slide over the blue Legendrian. Passing the cusps over the $n-1$ handle and cancelling the red with the blue, we obtain a max tb unknotted sphere in the complement of the $n-1$ handle. } 
    \label{fig:example_disk}
\end{figure}
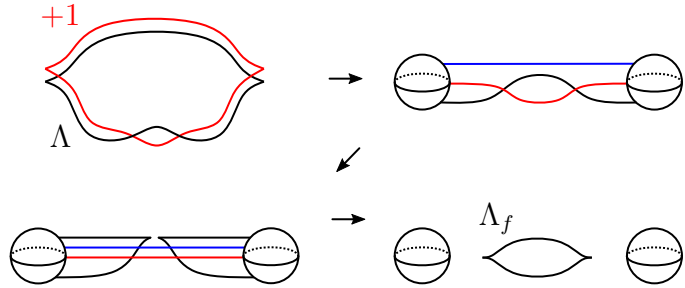

Here, $\Lambda_{f}$ is not loose, and is the Legendrian unknot in the subcritical domain $B^{2n} \cup H^{n-1}$. Indeed, we have constructed the Weinstein diagram of $T^*D^n \cup H^{n-1}$, i.e. $T^*D^n$ with a subcritical handle attached along a subcritical isotropic sphere in a Darboux chart. $\Lambda_+$ is the Legendrian unknot and bounds the Lagrangian unknot; carving out the Lagrangian unknot is equivalent to attaching an index $n-1$ handle. 
\end{examples} 

\begin{examples} 

Let $U \subset S^{n-1}$ be the disconnected union of codimension zero submanifold $U'$ and a disk $D^{n-1}$. In this case, after cancellation, we expect the remaining Legendrian to be loose. This is because $D_{U' \coprod D^{n-1}}^n$ is Lagrangian isotopic to  $D_{U'}^n \natural T^*_0 D^n,$
and by \cite{Lazarev_grothendieck_group}, $T^*D^n \setminus (D_{U'} \natural T^*_0 D^n)$ is obtained from the subcritical sector $T^*D^n \setminus (D_{U'} \coprod T^*_0 D^n)$
by attaching a flexible handle. 
We can see also see this from our explicit construction as follows. 

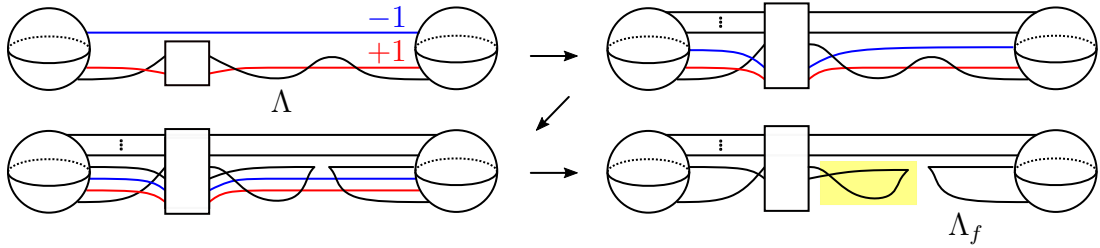
\begin{figure} [h!]
    \centering
    \begin{tikzpicture}
    \node at (0,0) {\includegraphics[width=14.5cm]{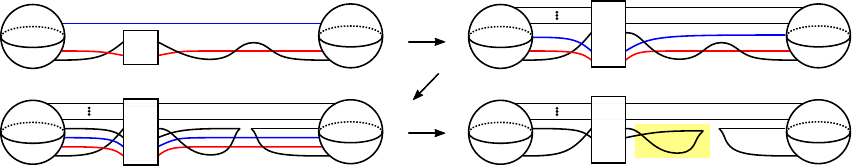}};
    \node[color=red] at (-2.2,0.75) {$+1$};
    \node[color=blue] at (-2.2,1.2) {$-1$};
    \node at (-3.6,0.1) {$\Lambda$};
    \node at (5.5,-1.6) {$\Lambda_f$};
    \end{tikzpicture}
    \caption{Our construction applied to the case where $U$ is a union of some submanifold (the rectangular region) and a disk (the maximum on the right). We first use boat moves and handleslides to remove all obstructing Reeb chords coming from the boxed region. A single handleslide then puts the red and blue Legendrians in a cancelling position. The yellow highlighted region is a loose chart.} 
    \label{fig:disconnected_example}
\end{figure}

Since $U$ is disconnected, we may isotope $\Lambda_{U}$ so that the Reeb chords coming from $g|_{U'}$, the $U'$-perturbation, are shorter than the Reeb chord coming from the maximum of $g|_{D^{n-1}}$, the $D^{n-1}$-perturbation (see Figure~\ref{fig:disconnected_example}). 

Next, we apply the construction of Theorem~\ref{thm: general_antisurgery_intro} to obtain the Legendrian attaching sphere $\Lambda_{f}$ as follows. A combination of boat moves and handleslides will first remove all the critical points coming from $U'$. Then, a single handleslide will remove the critical point coming from $D^{n-1}$ (see Figure~\ref{fig:disconnected_example}). We are now in position to cancel $\Lambda_-$ and $\Lambda_+$ as in the proof of Theorem~\ref{thm: general_antisurgery_intro}. 

After this cancellation, we see that the resulting Legendrian knot $\Lambda_f$ has a loose fishtail chart in a transverse slice.  By \cite{Murphy_12_LLEHD}, this implies that $\Lambda_f$ is loose.
\end{examples}

\begin{examples} 

When $0\in P$, the $P$-loose Legendrian $\Lambda_P$ is loose. To see this, first recall that $$D_U = \Gamma(df) \cap B^{2n} \text{ for a  function } f: D^n \rightarrow \mathbb{R}$$ 
which is an extension of a Morse function $f: S^{n-1} \to R$ that is negative on $U$ and positive on the closure of the complement. It is enough to prove that $\Lambda_P$ is loose for $P = \{0\}$; 
 as discussed in Section 2.2.2 of \cite{Lazarev_Sylvan_2023_PLWS}, $\Lambda_{P \coprod Q}$ is isotopic to $\Lambda_P \sharp \Lambda_Q$ and the connected sum of any Legendrian with a loose Legendrian (in a separate Darboux chart) is loose.

Let $U = S^k$, that is, $U$ is a $P$-Moore space for $P=\{0\}$ since its relative cohomology is $\mathbb{Z} \cong \mathbb{Z}/\{0\}$ in positive degree. Consider $U$ to be embedded in $S^{n-1}$ as the intersection $D^{k+1} \cap S^{n-1}$. Then, one can take the function $f: D^n \rightarrow \mathbb{R}$ to be a perturbation of
$$
 -(x^2_1 + \cdots x^2_{k+1}) + x_{k+2}^2 + \cdots x_{n}^2
$$

Then, $D_U = \Gamma(df)$ is Hamiltonian isotopic to the cotangent fiber, $T^*_0 D^n$, in $T^*D^n$. So, $T^*D^n \setminus T^*_0 D^n$ is just $B^{2n} \cup H^{n-1}$. 
Further, the stop $\partial D_U$ is a Legendrian that passes through the $H^{n-1}$ exactly one time, i.e., it is a loose Legendrian.

\begin{figure} [h!]
    \centering
    \begin{tikzpicture}
    \node at (0,0) {\includegraphics[width=14.5 cm]{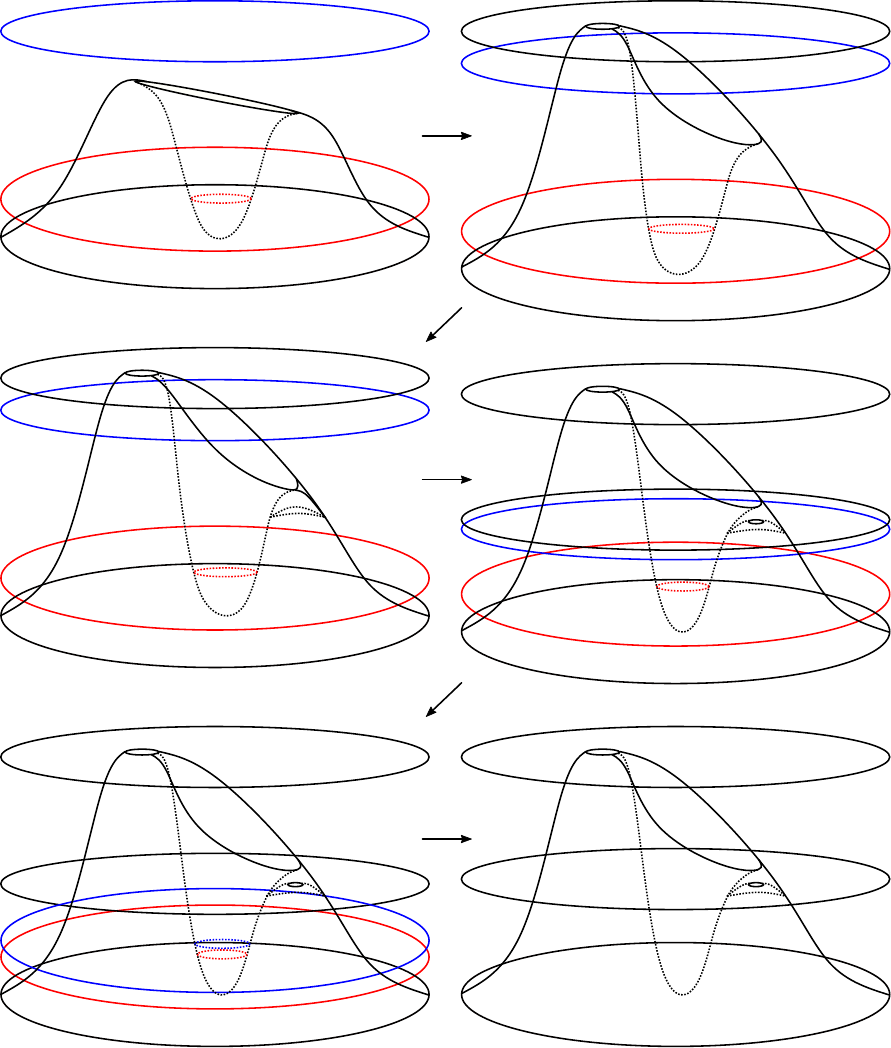}};
    \node[color=red] at (-7,6) {$\Lambda_+$};
    \node[color=blue] at (-7,7.4) {$\Lambda_-$};
    \node at (-7,4) {$\Lambda$};
    \node at (7,-6.7) {\textbf{H}};
    \end{tikzpicture}
    \caption{The process to obtain Diagram H when $n=3$. The black Legendrian $\Lambda$ begins with a maximum and a saddle point between the blue $\Lambda_-$ and the red  $\Lambda_+$. Away from the pictured region, $\Lambda_-$ and $\Lambda_+$ are parallel. We perform the following moves:\\\hspace{\textwidth}
			\textbf{(1.)} At the maximum, we perform a handleslide. \\\hspace{\textwidth}
			\textbf{(2.)} At the saddle point, we perform a boat move.\\\hspace{\textwidth}
			\textbf{(3.)} We then isotope $\Lambda_-$ downwards, and handleslide $\Lambda$ over $\Lambda_-$. \\\hspace{\textwidth}
			\textbf{(4.)} We further isotope $\Lambda_-$  so that it is in cancelling position with $\Lambda_+$. \\\hspace{\textwidth}
			\textbf{(5.)} Finally, we cancel $\Lambda_-$ and $\Lambda_+$, so that only $\Lambda$ remains.}
    \label{fig:example}
\end{figure}

\begin{figure} [h!]
    \centering
    \begin{tikzpicture}
    \node at (0,0) {\includegraphics[width=7.3 cm]{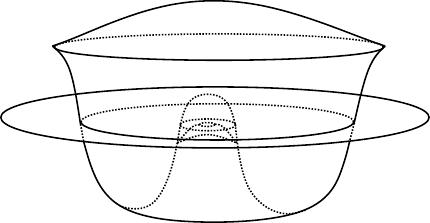}};
    \node at (-3.2,0.8) {$\Lambda_f$};
    \end{tikzpicture}
    \caption{After a Legendrian isotopy which passes the top circle of cusps across the 2-handle, the connected black Legendrian resulting from the recipe in Figure~\ref{fig:example} is in cancelling position with the subcritical handle whose attaching region is a torus around the middle disk in this dimension.} 
    \label{fig:example2}
\end{figure}

Next, we will explain how to see this explicitly by following the construction in the Proof \ref{prf: main construction}. We use the notation from Proof \ref{prf: main construction}. 
On $U = S^k$, $f$ has two critical points, namely, a maximum and a saddle point. Following the construction, to get an explicit Weinstein diagram we perform a single boat move and two handleslides. After cancelling 
$\Lambda_-$ and $\Lambda_+$, the resulting Legendrian (Diagram H) intersects the subcritical $(n-1)$-handle $H^{n-1}$ three times. This is because initially $\partial D^n$ intersected $H^{n-1}$ once, and each handleslide introduces one new intersection. For $n=3$, this process is illustrated in Figure~\ref{fig:example}. In higher dimensions, the construction follows analogously.

With Diagram H, instead of continuing with the construction in Proof \ref{prf: main construction} as is, we first do an additional step.
We perform a Legendrian isotopy that passes the circle of cusps over $H^{n-1}$. The resulting Legendrian, $\Lambda_f = \partial D_U$, passes through $H^{n-1}$ exactly one time.

To conclude, note that $\Lambda_f$ is now in cancelling position with $H^{n-1}$, see Figure~\ref{fig:example2}. Hence, it is loose, and remains loose when the flexible handle $H_{flex}$ is attached alongside it to cancel $H^{n-1}$. Thus, we obtain that $\Lambda_0$ is loose.
\end{examples}

\subsection{Questions}

One can construct a loose Legendrian unknot by pushing through any codimension zero subdomain $U$ (with boundary) past the Legendrian unknot \textit{near a cusp}.
As observed in \cite{Murphy_12_LLEHD},  $\Lambda_U$ is always loose, see Figure~\ref{fig:push_to_loose}. 
If the Euler characteristic of $U$ is 0, then $\Lambda_U$ is formally Legendrian isotopic to $\Lambda$ (but not genuinely Legendrian isotopic) and hence called the loose Legendrian unknot. So by the h-principle for loose Legendrians, $\Lambda_U$ and $\Lambda_V$ are isotopic if $\chi(U) = \chi(V)$. 

Our construction of the $P$-loose Legendrians also involves pushing through certain codimension zero subdomains (neighborhoods of $P$-Moore spaces). However, here the construction is less concrete; one must first push through to create the Lagrangian disk $D_U$, then carve out $D_U$, and then attach a flexible handle, ultimately resulting in our recipe above. Hence, it is natural to ask whether there is a more direct route towards the construction of these $P$-loose Legendrians, analogous to the construction to loose Legendrians by Murphy.
\begin{question}
Can a $P$-loose Legendrian unknot be constructed more directly by pushing through a $P$-Moore space past a region of the Legendrian unknot (not near a cusp), after Legendrian isotopy of the unknot?
\end{question}
For example, this pushing operation, if it exists, must have the property that if $U$ is disconnected, then $\Lambda_U$ must be loose, as discussed in Example~\ref{fig:example2}. 

Another line of inquiry is to see whether our algorithm can provide an alternative proof of the Ganatra-Pardon-Shende \cite{Ganatra_Pardon_Shende_descent} localization formula from the point of view of Legendrian invariants. 
The localization formula \cite{Ganatra_Pardon_Shende_descent} computes the wrapped Fukaya of the subdomain $X\setminus D$ as 
$$
\mathcal{W}(X\setminus D) \cong \mathcal{W}(X)/D,
$$
where $\mathcal{W}(X)/D$ is the algebraic localization of $\mathcal{W}(X)$ by the $D$. They describe a concrete formula computing the morphism chain complexes,  $\mathrm{Hom}_{\mathcal{W}(X)/D}(L, K)$, via a dg bar construction that depends on morphism $\mathrm{Hom}_{\mathcal{W}(X)}(L, K)$ as well as $\mathrm{Hom}_{\mathcal{W}(X)}(L, D)$ and $\mathrm{Hom}_{\mathcal{W}(X)}(D, K)$. For certain Lagrangians, like the co-cores of $X$, these complexes can all be computed using the Legendrian dga's of the attaching spheres of $X$ and $\partial D$. Hence, the Ganatra-Pardon-Shende surgery formula can be used a priori to  describe the Legendrian dga's of the attaching spheres of $X\setminus D$. On the other hand, the current paper given an explicit geometric Weinstein presentation for $X\setminus D$ and an explicit depiction of its Legendrian attaching spheres.  
\begin{question}
Can one compute the Legendrian dga's of the attaching spheres of $X\setminus D$ produced by Theorem \ref{thm: general_antisurgery_intro} directly, giving an alternative direct proof of the Ganatra-Pardon-Shende localization formula? 
\end{question}

\bibliographystyle{alpha}
\bibliography{references.bib}

\end{document}